\documentclass[11pt,leqno,oneside,letterpaper]{amsart}
\usepackage{amssymb,amsmath,amsfonts,latexsym, amscd, amsthm, graphicx, url, epsfig} 
\usepackage[letterpaper,left=1.75truein, top=1truein]{geometry}
\usepackage{color}
\usepackage[colorlinks,linkcolor=blue,citecolor=blue]{hyperref}
\newtheorem{theorem}{Theorem}[section]
\newtheorem{lemma}[theorem]{Lemma}
\newtheorem{proposition}[theorem]{Proposition}
\newtheorem{definition}[theorem]{Definition}

\newtheorem{corollary}[theorem]{Corollary} 

\theoremstyle{definition}
\newtheorem{remark}[theorem]{Remark}

\newtheorem{question}[theorem]{Question}

\newtheorem{addendum}[theorem]{Addendum}

\newtheoremstyle{cases}
  {12pt plus 6 pt}%       Space above
  {2pt}%       Space below
  {\bfseries}   %       Body font
  {}%          Indent amount (empty = no indent, \parindent = para indent)
  {\bfseries}% Thm head font
  {.}%         Punctuation after thm head
  {.5em}%      Space after thm head: " " = normal interword space;
  {}%          Thm head spec (can be left empty, meaning `normal')
\theoremstyle{cases}

\numberwithin{subcase}{case} \numberwithin{subsubcase}{subcase}
\numberwithin{equation}{subsection}

\begin{document}

\title[Slope detection, foliations in graph manifolds, and L-spaces]{Slope detection, foliations in graph manifolds, and L-spaces}

\author[Steven Boyer]{Steven Boyer}
\thanks{Steven Boyer was partially supported by NSERC grant RGPIN 9446-2008}
\address{D\'epartement de Math\'ematiques, Universit\'e du Qu\'ebec \`a Montr\'eal, 201 avenue du Pr\'esident-Kennedy, Montr\'eal, QC H2X 3Y7.}
\email{boyer.steven@uqam.ca}
\urladdr{http://www.cirget.uqam.ca/boyer/boyer.html}

\author{Adam Clay}
\thanks{Adam Clay was partially supported by NSERC grant RGPIN-2014-05465}
\address{Department of Mathematics, 342 Machray Hall, University of Manitoba, Winnipeg, MB, R3T 2N2.}
\email{Adam.Clay@umanitoba.ca}
\urladdr{http://server.math.umanitoba.ca/~claya/}

\begin{abstract}
A graph manifold rational homology $3$-sphere $W$ with a left-orderable fundamental group admits a co-oriented taut foliation, though it is unknown whether it admits a smooth co-oriented taut foliation. In this paper we extend the gluing theorem of \cite{BC} to graph manifold rational homology solid tori and use this to show that there are smooth foliations on the pieces of $W$ which come close to matching up on its JSJ tori. This is applied to prove that a graph manifold with left-orderable fundamental group is not an L-space. 
\end{abstract}

\maketitle
\vspace{-.6cm}
\begin{center}
\today
\end{center}

\section{Introduction}
Conjecture 1 of \cite{BGW} contends that an irreducible rational homology 3-sphere $W$ is {\it not} an L-space if and only if its fundamental group is left-orderable. In this paper we verify one direction of the conjecture for graph manifolds.   

\begin{theorem} \label{LO implies NLS}
A graph manifold L-space has a non-left-orderable fundamental group.
\end{theorem}

The question of whether a graph manifold rational homology $3$-sphere $W$ which has a non-left-orderable fundamental group is an L-space remains open, even in topologically simple cases, though it is shown in \cite{BB} that a graph manifold integer homology $3$-sphere is an L-space if and only if its fundamental group is not left-orderable.

In order to prove Theorem \ref{LO implies NLS}, we generalize some of the results of \cite{BC} concerning the foliation detection of slopes on the boundary of a Seifert manifold $M$  in several ways. (See \S \ref{sec: detection} for precise definitions.) Theorem \ref{detected slopes seifert} extends results on slope detection on a boundary component of $M$ relative to a fixed slope on each of its remaining boundary components to slope detection relative to certain fixed families of slopes on each of its remaining boundary components. This is used to study slope detection on the boundary of a graph manifold rational homology solid torus. 

\begin{theorem} \label{detection in graph manifolds qhst}
Let $V$ be an orientable graph manifold rational homology solid torus and let $\mathcal{D}(V)$ be the set of foliation detected slopes on $\partial V$ and $\mathcal{D}_{str}(V)$ the set of strongly foliation detected slopes. Then $\mathcal{D}(V)$ is a non-empty, closed, connected subset of the circle $\mathcal{S}(\partial V)$ whose frontier consists of rational slopes. Further, $\mathcal{D}(V) \setminus \mathcal{D}_{str}(V)$ is finite and consists of rational slopes.  
\end{theorem}
A more precise version of this theorem can be found in \S \ref{subsec: slope detection graph}. In  \S \ref{sec: degenerate} we consider the degenerate case when the set of detected slopes on $V$ has a unique element, and then apply the material from Sections \ref{sec: slope detection graph} and \ref{sec: degenerate} to show that if a graph manifold rational homology $3$-sphere $W$ admits a co-oriented taut foliation, then we can find smooth foliations on its pieces which come close to matching up on the JSJ tori of $W$. 

\begin{theorem} \label{block decomposition} 
Let $W$ be a graph manifold  rational homology $3$-sphere which admits a co-oriented, taut foliation. We can split $W$ along a subset $\mathcal{T}_0$ of its JSJ tori into connected graph manifolds $U_1, \ldots, U_n$ such that 

$(1)$ each $U_i$ admits a smooth, co-oriented, taut, boundary-transverse, codimension $1$ foliation $\mathcal{F}_i$  which intersects the components of $\partial U_i$ in hyperbolic foliations; 

$(2)$ if $U_i$ and $U_j$ are incident to a torus $T \in \mathcal{T}_0$, then the rational slopes on $T$ determined by $\mathcal{F}_i$ and $\mathcal{F}_j$ coincide. 

\end{theorem}

Theorem \ref{block decomposition} combines with standard gluing techniques from contact geometry to establish Theorem \ref{LO implies NLS}. 

It was pointed out in \cite[\S 12.2]{BC} that the generic graph manifold rational homology $3$-sphere which admits a co-oriented taut foliation also admits a smooth co-oriented taut foliation, though the general case remains open.  

\begin{question}
Let $W$ be a graph manifold  rational homology $3$-sphere which admits a co-oriented, taut foliation. Does $W$ admits a smooth co-oriented taut foliation? 
\end{question}
The material of the paper, and in particular \S \ref{sec: block decompositions}, can be considered an initial analysis of this question. In particular, see Question \ref{q: matching hyperbolic behaviour}. 

Background material is developed in \S \ref{sec: detection} and a relative slope detection theorem for Seifert manifolds is proved in \S \ref{sec: relative detection}. Sections \ref{s: horizontal foliations} and \ref{sec: operations} develop some technical results on smooth foliations in Seifert manifolds which lead to the proof of Theorem \ref{detection in graph manifolds qhst} in \S \ref{sec: slope detection graph}. In \S \ref{sec: degenerate} we study foliations on graph manifold rational homology solid tori with degenerate sets of detected slopes. We prove Theorem \ref{block decomposition} in \S \ref{sec: block decompositions}. Finally, in \S \ref{sec: lo --> nls} we show why Theorem \ref{block decomposition} implies that $W$ is not an L-space, and conclude that Theorem \ref{LO implies NLS} holds.  

{\bf Acknowledgements}. The authors would like to thank Jonathan Bowden who pointed out a gap in our original proof of Theorem \ref{LO implies NLS} and who suggested the technique needed to fill it. 

\section{Relative slope detection in Seifert manifolds} \label{sec: relative detection} \label{sec: detection} 

The goal of this section is to present an extension (Theorem \ref{detected slopes seifert}) of some of the results of the Appendix of \cite{BC}. 

\subsection{Rational longitudes}

A {\it rational longitude} of a compact, connected, orientable $3$-manifold $N$ with boundary a torus is a primitive class $\lambda_N \in H_1(\partial N)$ which represents a torsion element when considered as an element of $H_1(N)$. Rational longitudes exist and are well-defined up to sign. 

\subsection{Orientable graph manifold rational homology solid tori} \label{pieces}

Throughout the paper $V$ will denote an orientable graph manifold rational homology solid torus, $P$ a punctured $2$-sphere, $Q$ a punctured projective plane and $Q_0$ a M\"{o}bius band. A JSJ piece of $V$ is a compact, connected, orientable Seifert manifold $M$, distinct from $S^1 \times D^2$ and $S^1 \times S^1 \times I$, whose boundary is a non-empty union of tori $T_1, \ldots , T_r$ and whose base orbifold is of the form $P(a_1, a_2, \ldots, a_n)$ or $Q(a_1, a_2, \ldots, a_n)$ where $n \geq 0$ and $a_1, \ldots , a_n \geq 2$. The Seifert fibring on $M$ is unique up to isotopy unless $M$ is a twisted $I$-bundle over the Klein bottle, denoted $N_2$, which admits exactly two isotopy classes of Seifert structures. One has base orbifold $Q_0$ and the other has base orbifold $D^2(2,2)$. Let $h_0, h_1 \in H_1(\partial N_2)$ denote, respectively, primitive classes carried by a Seifert fibre of the structure with base orbifold $Q_0$, respectively $D^2(2,2)$. Then $\{h_0, h_1\}$ is a basis of $H_1(\partial N_2)$ well-defined up to sign change of $h_0$ or $h_1$. 

When $M \not \cong N_2$ the class of a regular Seifert fibre of $M$ is well-defined up to taking inverses and we use $h \in \pi_1(M)$ to denote it. For each boundary component $T_j$ of $M$ we will also use $h$ to denote a primitive class of $H_1(T_j)$ represented by a Seifert fibre. When $M \cong N_2$, $h$ will correspond to either $h_0$ or $h_1$, depending on the Seifert structure chosen for $M$. The rational longitude of $N_2$ is represented by $h_0$.

\subsection{Slopes}  \label{slopes} A {\it slope} on a torus $T$ is the class $[\alpha]$ of a non-zero element $\alpha \in H_1(T; \mathbb R)$ in the projective space 
$$\mathcal{S}(T) = P(H_1(T; \mathbb R)) \cong S^1$$ 
We call a slope on $T$ {\it rational} if it is represented by a class $\alpha \in H_1(T)$. Otherwise we call it {\it irrational}. For instance, a rational longitude $\lambda_N$ of a compact, connected, orientable $3$-manifold $N$ with boundary a torus determines a well-defined rational slope $[\lambda_N] \in \mathcal{S}(\partial N)$.

More generally, given a finite collection of tori $\mathcal{T}= \{T_1, T_2, \ldots ,T_m\}$ we set
$$\mathcal{S}(\mathcal{T}) = \mathcal{S}(T_1) \times \ldots \times \mathcal{S}(T_m)$$

\subsection{Foliation detection of slopes} \label{fol detect}
Let $N$ be a compact, connected, orientable $3$-manifold whose boundary is a union of tori $T_1, T_2, \ldots, T_r$. Set $\mathcal{T}(\partial N) = \{T_1, \ldots, T_r\}$ and define  
$$\mathcal{S}(N) = \mathcal{S}(\mathcal{T}(\partial N)) =  \mathcal{S}(T_1) \times \ldots \times \mathcal{S}(T_r)$$

Each co-oriented taut foliation $\mathcal{F}$ on $N$ which is transverse to $\partial N$ determines a foliation $\mathcal{F} \cap T_j$ on $T_j$ which is the suspension of a homeomorphism $f_j$ of a circle contained in $T_j$. As such it determines a slope $[\alpha_j(\mathcal{F})]$ on $T_j$ (cf. \cite[\S 6]{BC}). For instance, if $\mathcal{F} \cap T_j$ contains a closed leaf, then $[\alpha_j(\mathcal{F})]$ is the slope of that leaf. If $f_j$ is conjugate in $\hbox{Homeo}_+(S^1)$ to a rotation, we call $\mathcal{F} \cap T_j$ {\it linear}.

\begin{definition}
{\rm Let $\mathcal{F}$ be a co-oriented taut foliation on a compact, connected, orientable $3$-manifold $N$ (as above) which is transverse to $\partial N$.  
A slope $[\alpha_j] \in \mathcal{S}(T_j)$ is {\it detected} by $\mathcal{F}$, or $\mathcal{F}$-{\it detected}, if $[\alpha] = [\alpha_j(\mathcal{F})]$. It is {\it strongly $\mathcal{F}$-detected} if it is $\mathcal{F}$-detected and $\mathcal{F} \cap T_j$ is linear. For $J \subset \{1, 2, \ldots, r\}$ and $[\alpha_*] = ([\alpha_1], [\alpha_2], \ldots , [\alpha_r]) \in \mathcal{S}(N)$, we say that $(J; [\alpha_*])$ is $\mathcal{F}$-{\it detected} if $\mathcal{F}$ detects $[\alpha_j]$ for all $j$ and $\mathcal{F}$ strongly detects $[\alpha_j]$ for $j \in J$. Finally, we say that $(J; [\alpha_*])$ is {\it foliation}-{\it detected} if it is $\mathcal{F}$-detected for some co-oriented taut foliation $\mathcal{F}$ on $N$.
}
\end{definition} 

Let 
$$\mathcal{D}(N ; J) = \{[\alpha_*] \in \mathcal{S}(N) : (J; [\alpha_*]) \hbox{ is foliation detected}\}$$ 
When $J = \emptyset$, we will often write $\mathcal{D}(N) $ in place of $\mathcal{D}(N ; J)$.

\subsection{Relative slope detection in Seifert manifolds} \label{sec: relative detection}
For the remainder of this section we take $M$ to be a Seifert manifold with base orbifold $Q(a_1, \ldots, a_n)$ or $P(a_1, \ldots, a_n)$ and boundary $T_1 \cup \ldots \cup T_r$ where $r \geq 1$ and $n \geq 0$ (cf. \S \ref{pieces}). 

We call $[\alpha_*]  \in \mathcal{S}(M)$ {\it horizontal} if no $[\alpha_j]$ coincides with the slope of the fibre class $[h]$.  

For $[\alpha_*] \in \mathcal{S}(M)$, $v([\alpha_*])$ will denote the number of vertical $[\alpha_j]$: 
$$v([\alpha_*]) = |\{j : [\alpha_j] = [h]\}|$$ 
Thus $[\alpha_*]$ is horizontal if and only if $v([\alpha_*]) = 0$. 

Fix $J \subseteq \{1, 2, \ldots , r-1\}$ and non-empty, closed, connected subsets $S_j$ of $\mathcal{S}(T_j)$, $1 \leq j \leq r-1$ such that $[h] \not \in S_j$ for $j \in J$. Assume that each $S_j$ has rational endpoints (if any). Set $S_* = S_1 \times \ldots \times S_{r-1}$ and for each $[\alpha_*] \in S_*$ define sets of relatively detected slopes 
$$\mathcal{D}(M; T_r; J; [\alpha_*]) = \{[\alpha_r] \in \mathcal{S}(T_r) : (J; ([\alpha_1], \ldots, [\alpha_{r-1}], [\alpha_r])) \hbox{ is foliation detected}\}$$
$$\mathcal{D}_{str}(M; T_r; J; [\alpha_*])  = \{[\alpha_r] \in \mathcal{S}(T_r) : (J \cup\{r\}; ([\alpha_1], \ldots, [\alpha_{r-1}], [\alpha_r])) \hbox{ is foliation detected}\}$$
Now set  
$$\mathcal{D}(M;T_r; J;S_*)  = \bigcup_{[\alpha_*] \in S_*} \mathcal{D}(M;T_r; J; [\alpha_*]) \subseteq \mathcal{S}(T_r)$$
$$\mathcal{D}_{str}(M;T_r; J;S_*)  = \bigcup_{[\alpha_*] \in S_*} \mathcal{D}_{str}(M;T_r; J; [\alpha_*]) \subseteq \mathcal{S}(T_r)$$ 
When $J = \emptyset$, we will often drop $J$ from the notation and write $\mathcal{D}(M; T_r; [\alpha_*])$ and $\mathcal{D}_{str}(M;T_r; S_*)$. We determine various basic properties of these sets below. First we introduce some notation and normalizations. 

For each $j$ such that $|S_j| = 1$ we have $S_j = \{[\alpha_j]\}$ where $[\alpha_j]$ is rational. If, in this case, $j \in J$, we can replace $M$ by $M' = M([\alpha_j])$ and $J$ by $J \setminus \{j\}$ without changing the sets $\mathcal{D}_{str}(M; T_r; J; S_*)$ and $\mathcal{D}(M; T_r; J; S_*)$. Note though that after all such operations $M$ might become a solid torus or the product of a torus and an interval, contrary to the assumptions in \cite{BC}. Under these extended possibilities for $M$ we suppose that 
$$|S_j| > 1 \hbox{ whenever } j \in J$$   

We  also assume that the boundary components of $M$ are indexed so that $J = \{1, \ldots, s\}$ and $|S_j| > 1$ if and only if $1 \leq j \leq t$ for some $t \geq s$.  Set $v( S_*) = |\{ j : [h] \in S_j \}|$.

\begin{theorem} \label{detected slopes seifert} 
Let $M$ and $S_*$ be as above and recall that $[h] \not \in S_j$ for $j \in J$\footnote{If $[\alpha_j] = [h]$ for some $j \in J$, then either $\mathcal{D}(M; T_r; J; S_*) = \emptyset$ or $M$ is the product of a torus and an interval. In this case $\mathcal{D}_{str}(M;T_r; J; S_*) = \mathcal{D}(M; T_r; J; S_*) \equiv S_1$}.  

$(1)$ If $M$ has base orbifold $Q(a_1, \ldots, a_n)$ then  

\indent \hspace{.5cm} $(a)$ $\mathcal{D}(M; T_r; J; S_*) = 
\left\{ \begin{array}{ll} \{[h]\} & \hbox{ if } v(S_*) = 0 \\ & \\ \mathcal{S}(T_r) & \hbox{ if } v(S_*) > 0 \end{array} \right.$

\indent \hspace{.5cm} $(b)$ $\mathcal{D}_{str}(M;T_r; J; S_*) = 
\left\{ \begin{array}{ll}  \{[h]\} & \hbox{ if } M \cong N_2 \\ & \\ \emptyset &  \hbox{ if } v(S_*) = 0 \hbox{ and } M \not \cong N_2 \\ & \\ \mathcal{S}(T_r) \setminus \{[h]\} & \hbox{ if } v(S_*) > 0 \end{array} \right.$

$(2)$ If $M$ has base orbifold $P(a_1, \ldots, a_n)$ and  
\begin{itemize}

\item $v(S_*) = 0$ then $\mathcal{D}(M; T_r; J; S_*)$ is a non-empty interval of horizontal slopes with rational endpoints. Further either 

\begin{itemize}

\vspace{.3cm} \item $M \cong S^1 \times S^1 \times I$ and $\mathcal{D}_{str}(M;T_r; J; S_*) = \mathcal{D}(M; T_r; J; S_*) \cong S_1$, or  

\vspace{.3cm} \item $M \not \cong S^1 \times S^1 \times I$ and $\mathcal{D}_{str}(M;T_r; J; S_*) = \mathcal{D}(M; T_r; J; S_*)$ consists of a single rational slope, or 

\vspace{.3cm} \item $M \not \cong S^1 \times S^1 \times I$ and $(\mathcal{D}(M; T_r; J; S_*), \mathcal{D}_{str}(M;T_r; J; S_*)) \cong ([0,1], (0,1))$.
\end{itemize}

\vspace{.3cm} \item $v(S_*) = 1$ then $\mathcal{D}(M; T_r; J; S_*)$ is a closed, connected subset of $\mathcal{S}(T_r)$, with rational endpoints $($if any$)$, containing $\{[h]\}$. Further, either

\begin{itemize}

\vspace{.3cm} \item $M \cong S^1 \times S^1 \times I$ and $\mathcal{D}_{str}(M;T_r; J; S_*) = \mathcal{D}(M; T_r; J; S_*) = S_1$, or  

\vspace{.3cm} \item $M \not \cong S^1 \times S^1 \times I$ and $\mathcal{D}_{str}(M;T_r; J; S_*) = \mathcal{D}(M; T_r; J; S_*) \setminus \{[h]\}$\footnote{If $\mathcal{D}(M; T_r; J; S_*)$ is homeomorphic to a non-degenerate interval, then $\mathcal{D}_{str}(M;T_r; J; S_*)$ has two connected components.}.  
\end{itemize}

\vspace{.3cm} \item $v(S_*) > 1$ then $\mathcal{D}(M; T_r; J; S_*) = \mathcal{S}(T_r)$. Further either 

\begin{itemize}

\vspace{.3cm} \item $M \cong S^1 \times S^1 \times I$ and $\mathcal{D}_{str}(M;T_r; J; S_*) = \mathcal{D}(M; T_r; J; S_*) = \mathcal{S}(T_r)$, or  

\vspace{.3cm} \item $M \not \cong S^1 \times S^1 \times I$ and $\mathcal{D}_{str}(M;T_r; J; S_*) = \mathcal{D}(M; T_r; J; S_*)  \setminus \{[h]\} = \mathcal{S}(T_r) \setminus \{[h]\}$. 
\end{itemize}

\end{itemize}

\end{theorem}

\begin{proof} 
If $M$ has base orbifold $Q(a_1, \ldots, a_n)$ then assertion (1) follows from Proposition 6.5 and Lemma 6.4 of \cite{BC}. The same result implies that assertion (2) holds when $M$ has base orbifold $P(a_1, \ldots, a_n)$ and $v(S_*) > 1$. To complete the proof we must deal with the case that $M$ has base orbifold $P(a_1, \ldots, a_n)$ and $v(S_*) \leq 1$. 

First suppose that $v(S_*) = 0$.  For $[\alpha_j] \in \mathcal{S}(T_j)$ define $\tau([\alpha_j]) \in \mathbb{R}$ according to $[\alpha] = [\tau([\alpha_j]) h - h_j^*]$.  For $[\alpha_*] \in \mathcal{S}(M)$ set $\tau([\alpha_*]) = (\tau([\alpha_1]), \ldots, \tau([\alpha_r]))$ and observe that this correspondence is bijective. Set 
$$\tau(S_*) = \{\tau([\alpha_*])  \in \mathbb R^{r-1} : [\alpha_*] \in S_*\}$$ 
$$\mathcal{T}(M; T_r; J; \tau_*) = \{\tau([\alpha]) : [\alpha] \in \mathcal{D}(M; T_r; J; \tau^{-1}(\tau_*))\}$$ 
and
$$\mathcal{T}_{str}(M; T_r; J; \tau_*) = \{\tau([\alpha]) : [\alpha] \in \mathcal{D}_{str}(M; T_r; J; \tau^{-1}(\tau_*))\}$$ 
As in \cite{BC} we can identify $\mathcal{D}(M; T_r; J; S_*)$ with 
$$\mathcal{T}(M;T_r; J; S_*) = \bigcup_{\tau_* \in \tau(S_*)} \mathcal{T}(M; T_r; J; \tau_*) \subset \mathbb R$$
and $\mathcal{D}_{str}(M;T_r; J; S_*)$ with 
$$\mathcal{T}_{str}(M;T_r; J; S_*) = \bigcup_{\tau_* \in \tau(S_*)} \mathcal{T}_{str}(M; T_r; J \cup \{r\}; \tau_*) \subset \mathbb R$$
Under this identification we set 
$$T_j = \tau(S_j) = [\eta_j, \zeta_j]$$ 
where $\eta_j, \zeta_j \in \mathbb Q$. 

Recall that for $\tau_* \in \mathbb R^{r-1}$, the following notation was introduced in \cite[Appendix]{BC}. 
\begin{itemize}
\item $r_1(\tau_*) = | \{ j  : \tau_j \notin \mathbb{Z} \} |$, the number of non-integral $\tau_j$;

\item $s_0(\tau_*) = | \{ j : \tau_j \in \mathbb{Z} \hbox{ and } j \notin J \}|$, the number of integral $\tau_j$ whose indices are not in $J$; 

\item $i_0(\tau_*) = | \{ j : \tau_j \in \mathbb{Z} \hbox{ and } j \in J \}|$, the number of integral $\tau_j$ whose indices are in $J$;

\item $b_0(\tau_*) = -(\lfloor \tau_1 \rfloor + \ldots + \lfloor \tau_{r-1} \rfloor )$;

\item $m_0(\tau_*) = b_0(\tau_*) - (n + r_1(\tau_*) + s_0(\tau_*) -1) = b_0(\tau_*) +  i_0(\tau_*) - (n + r - 1)$; 

\item  $m_1(\tau_*) = b_0(\tau_*) + s_0(\tau_*) -1$; 

\end{itemize}

{\bf Case 1}. $n + r \leq 2$. 

If $r = 1$ then $M$ is a solid torus. It follows that $\mathcal{T}_{str}(M;T_r; J; S_*) = \mathcal{T}(M;T_r; J; S_*)$ consists of the meridional slope of $M$. Thus we are done. If $r = 2$ then $M \cong S^1 \times S^1 \times I$, and $\mathcal{T}_{str}(M;T_r; J; S_*) = \mathcal{T}_{str}(M;T_r; J; S_*) \cong S_1$.  

{\bf Case 2}. $n + r \geq 3$. 

It follows from \cite[Proposition A.4]{BC} that for all $\tau_* \in \tau(S_*)$, $[m_0(\tau_*), m_1(\tau_*)] \subseteq \mathcal{T}(M;T_r; J; S_*)$, and 
$$\mathcal{T}(M;T_r; J; S_*) \subset \bigcup_{\tau_* \in \tau(S_*)} (m_0(\tau_*) - 1, m_1(\tau_*) + 1)$$
where $m_0(\tau_*) = b_0(\tau_*) - ((n + r - 1) - i_0(\tau_*))$ and $m_1(\tau_*) = b_0(\tau_*) + s_0(\tau_*) -1$. Since  $b_0(\tau_*), s_0(\tau_*)$, and $i_0(\tau_*)$ vary over a finite set of integers, there are integers $c_{min}$ and $c_{max}$ such that 
$$\{c_{min}, c_{max}\} \subseteq \mathcal{T}_{str}(M;T_r; J; S_*) \subset (c_{min} - 1, c_{max} + 1)$$ 
We prove that $[c_{min}, c_{max}] \subseteq \mathcal{T}(M;T_r; J; S_*)$ first. 

If $i_0(\tau_*) > 0$ there is some $j \leq s$ such that $\tau_j$ is an integer. If $\tau_j \in [\eta_j, \zeta_j)$ we can find $\tau_j' \in (\tau_j, \zeta_j) \setminus \mathbb Z$ such that if $\tau_*' = (\tau_1, \ldots,  \tau_{j-1}, \tau_j', \tau_{j+1}, \ldots, \tau_{r-1})$, then $b_0(\tau_*') = b_0(\tau_*), s_0(\tau_*') = s_0(\tau_*)$, and $i_0(\tau_*') = i_0(\tau_*) - 1$. Then $m_0(\tau_*') = m_0(\tau_*) - 1$ while $m_1(\tau_*') = m_1(\tau_*)$. On the other hand if $\tau_j = \zeta_j$, there is a $\tau_j' \in (\eta_j, \zeta_j)$ such that if $\tau_*' = (\tau_1, \tau_{j-1}, \tau_j', \tau_{j+1}, \ldots, \tau_{r-1})$, then $b_0(\tau_*') = b_0(\tau_*) + 1, s_0(\tau_*') = s_0(\tau_*)$, and $i_0(\tau_*') = i_0(\tau_*) - 1$. In this case, $b_0(\tau_*') + i_0(\tau_*') = b_0(\tau_*) + i_0(\tau_*)$ so $m_0(\tau_*) = m_0(\tau_*')$ and $m_1(\tau_*') = m_1(\tau_*) + 1$. Thus for each $\tau_* \in \tau(S_*)$ with $i_0(\tau_*) > 0$, there is a $\tau_*' \in \tau(S_*)$ such that $i_0(\tau_*') < i_0(\tau_*)$, $m_0(\tau_*')  \leq m_0(\tau_*)$ and $m_1(\tau_*') \geq m_1(\tau_*)$. Proceeding inductively we can show that for each $\tau_* \in \tau(S_*)$ with $i_0(\tau_*) > 0$, there is a $\tau_*' \in \tau(S_*)$ such that $\tau_j' = \tau_j$ for $s+1 \leq j \leq r-1$, $i_0(\tau_*') = 0$ and $[m_0(\tau_*), m_1(\tau_*)] \subseteq [m_0(\tau_*'), m_1(\tau_*')]$. Thus it suffices to show that 
$$[c_{min}, c_{max}] = \bigcup_{\tau_* \in \tau(S_*)  \ni i_0(\tau_*) = 0} [m_0(\tau_*), m_1(\tau_*)]$$

Fix $\tau^0_* \in \tau(S_*)$ such that $i_0(\tau^0_*) = 0$, then $m_1(\tau^0_*) - m_0(\tau^0_*) = n + r + s_0(\tau_*) - 2 \geq s_0(\tau_*) + 1 \geq 1$. Consider the sets $[m_0(\tau_*), m_1(\tau_*)]$ with $i_0(\tau_*) =0$ as $\tau_j$ varies. For $1 \leq j \leq s$, both $m_1(\tau_*)$ and $m_0(\tau_*)$ are locally constant as $\tau_j$ varies over $\tau(S_j) \setminus \mathbb{Z}$ except for a finite number of jumps of size $1$. For $s+1 \leq j \leq t$, as $\tau_j$ varies over $\tau(S_j)$ both $m_0(\tau_*)$ and $m_1(\tau_*)$ are locally constant except for a finite number of jumps of size $1$, though the jumps of $m_0(\tau_*)$ are disjoint from those of $m_1(\tau_*)$. It follows that for each $j$, $\bigcup_{\tau_j \in \tau(S_j)  \ni i_0(\tau_*) = 0} [m_0(\tau_*), m_1(\tau_*)] $ is connected and contains $[m_0(\tau^0_*), m_1(\tau^0_*)]$. Thus $\bigcup_{\tau_* \in \tau(S_*)  \ni i_0(\tau_*) = 0} [m_0(\tau_*), m_1(\tau_*)] $ is connected and therefore 
$$\bigcup_{\tau_* \in \tau(S_*)  \ni i_0(\tau_*) = 0} [m_0(\tau_*), m_1(\tau_*)] = [c_{min}, c_{max}] \subseteq \mathcal{T}_{str}(M;T_r; J; S_*) \subset (c_{min} - 1, c_{max} + 1)$$ 

Suppose that $\mathcal{T}_{str}(M;T_r; J; S_*) = [c_{min}, c_{max}]$. If $\{c_{min}, c_{max}\} \cap \mathcal{T}_{str}(M;T_r; J; S_*) \ne \emptyset$, then \cite[Proposition A.4]{BC} implies that there is some $\tau_* \in \tau(S_*)$ such that $i_0(\tau_*) = 0$ and $\{m_0(\tau_*)\} = \mathcal{T}(M;T_r; J; \tau_*) = \mathcal{T}_{str}(M;T_r; J; \tau_*)$, contrary to the fact that $m_1(\tau_*) - m_0(\tau_*) \geq 1$. Thus \cite[Proposition A.4]{BC} implies that $\mathcal{T}_{str}(M;T_r; J; S_*) = (c_{min}, c_{max})$.

Next suppose that $\mathcal{T}_{str}(M;T_r; J; S_*) \ne [c_{min}, c_{max}]$. It is not hard to see that  
\begin{eqnarray} 
\label{cmin}
c_{min} = \min\{b_0(\tau_*) - (n + r - 1) : \tau_* \in \tau(S_*) \ni i_0(\tau_*) = 0\} = i_1 - \sum_{j=1}^{r-1} \lfloor \zeta_j \rfloor 
\end{eqnarray} 
where  $i_1 =  |\{j : 1 \leq j \leq s \hbox{ and } \zeta_j \in \mathbb Z\}|$. Similarly
\begin{eqnarray}  
\label{cmax} 
c_{max} = \max\{b_0(\tau_*) + s_0(\tau_*) - 1 :  \tau_* \in \tau(S_*) \ni i_0(\tau_*) = 0\} = s_1 - 1 -\sum_{j = 1}^{r-1} \lfloor \eta_j \rfloor
\end{eqnarray}
where $s_1 = |\{j : s+1 \leq j \leq r-1 \hbox{ and } \eta_j \in \mathbb Z\}|$. 

Suppose that there is some $\tau_* \in \tau(S_*)$ such that $(c_{min} - 1, c_{min}) \cap \mathcal{T}(J; \tau_*) \ne \emptyset$.  Then there is a rational number $\eta(\tau_*) \in (c_{min}-1, c_{min})$ such that 
$(c_{min} - 1, c_{min}] \cap \mathcal{T}(J; \tau_*) = [\eta(\tau_*), c_{min}]$ and by \cite[Theorem A.1]{BC}, this implies that that $s_0(\tau_*) =0$.  Without loss of generality we can suppose that for $j \in J$, $\tau_j \in \mathbb Z$ if and only if $1 \leq j \leq i_0(\tau_*)$.  Moreover, in this case we can assume that $\tau_j = \zeta_j$ as otherwise we would have $m_0(\tau_*) > c_{min}$ by (\ref{cmin}). By \cite[Appendix]{BC} there are coprime integers $0 < A < N$ and a permutation $(\frac{A_1}{N}, \ldots, \frac{A_n}{N}, \frac{B_{i_0(\tau_*) +1}}{N}, \ldots, \frac{B_{r-1}}{N},  \frac{C}{N})$ of $(\frac{A}{N}, 1 - \frac{A}{N}, \frac{1}{N}, \ldots , \frac{1}{N})$ such that  
\begin{enumerate}

\vspace{.3cm} \item  $1 - \frac{A_i}{N} < \gamma_i \hbox{ for } 1 \leq i \leq n$; 

\vspace{.2cm} \item $1 - \frac{B_{j}}{N} < \tau_j-\left \lfloor{\tau_j}\right \rfloor $ for all $i_0(\tau_*) + 1 \leq j \leq s$ and $1 - \frac{B_{j}}{N} \leq \tau_j-\left \lfloor{\tau_j}\right \rfloor$ for all $s+1 \leq j \leq r-1$;

\vspace{.2cm} \item $1 - \frac{C}{N} = \eta(\tau_*) -(c_{min}-1)$.

\end{enumerate}
Now choose $\tau_1', \ldots, \tau_{i_0(\tau_*)}'$ such that $\tau_j' \in  (\zeta_j-1 , \zeta_j) \cap \tau(S_j)$ and $1 - \frac{1}{N} < \tau_j' - \left \lfloor{\tau_j'}\right \rfloor$ for $1 \leq j \leq i_0(\tau_*)$ and define $\tau_*'$ by replacing the $\tau_j$ by $\tau_j'$ for $1 \leq j \leq i_0(\tau_*)$. Then the reader will verify that $i_0(\tau_*') = 0$ and $\eta(\tau_*') \leq \eta(\tau_*)$. Hence, 
$$\mathcal{T}_{str}(M;T_r; J; S_*) \cap (c_{min} - 1, c_{min}) = \bigcup_{\tau_* \in \tau(S_*)  \ni i_0(\tau_*) = 0} \big(\mathcal{T}(M;T_r; J; \tau_*)  \cap (c_{min} - 1, c_{min})  \big)$$

Suppose then that $(c_{min} - 1, c_{min}) \cap \mathcal{T}(M;T_r; J; \tau_*) \ne \emptyset$ where $i_0(\tau_*) = 0$, then $c_{min} = b_0(\tau_*) - (n + r - 1)$. If $\zeta_j \in \mathbb Z$ for some $s + 1 \leq j \leq r-1$ then we must have $\tau_j = \zeta_j$ by (\ref{cmin}), and therefore $s_0(\tau_*) > 0$, a contradiction. Thus $\zeta_j \not \in \mathbb Z$ for $s + 1 \leq j \leq r-1$ and $\tau_j \in \tau(S_j) \cap(\lfloor \zeta_j \rfloor, \zeta_j)$ for such $j$. The reader will verify from the proof of Subcase 2.1 of the proof of \cite[Proposition A.4]{BC} that $(c_{min} - 1, c_{min}] \cap \mathcal{T}(M;T_r; J; \tau_*) \subseteq (c_{min} - 1, c_{min}] \cap \mathcal{T}(M;T_r; J; \zeta_*)$ where $\zeta_* = (\zeta_1, \ldots, \zeta_{r-1}) \in \tau(S_*)$. Since $\mathcal{T}(M;T_r; J; \zeta_*) \cap (c_{min}-1, c_{min}] = [\eta(\zeta_*), c_{min}]$ it follows that $\mathcal{T}_{str}(M;T_r; J; S_*) \cap (c_{min} - 1, c_{min}] = [\eta(\zeta_*), c_{min}]$. Similarly $\mathcal{T}_{str}(M;T_r; J; S_*) \cap (c_{min} - 1, c_{min}] = (\eta(\zeta_*), c_{min}]$.

The case where $\mathcal{T}_{str}(M;T_r; J; S_*) \cap [c_{max}, c_{max} + 1) \ne \emptyset$ can be handled in the same way to show that $\mathcal{T}_{str}(M;T_r; J; S_*) \cap [c_{max}, c_{max} + 1) = [c_{max}, \zeta(S_*)]$ and $\mathcal{T}_{str}(M;T_r; J; S_*) \cap [c_{max}, c_{max} + 1) = [c_{max}, \zeta(S_*))$ for some rational number $\zeta(S_*)$. This completes the proof when $v(S_*) = 0$.

The case that $v(S_*) = 1$ follows by applying the arguments of the case $v(S_*) = 0$ to $S_*^h := \{[\alpha_*] \in S_* : v([\alpha_*]) = 0\}$ and observing that $[h] \in \mathcal{S}(T_r)$ is detected. In particular if $[h] \in S_j$ then $\tau(S_j \setminus \{ [h] \}) = (-\infty, a] \cup [b, \infty)$ for some $a, b \in \mathbb{R}$.  Choose $N \in \mathbb{N}$ such that $-N<a$ and $b<N$, then for $n>N$ set 
\[\mathcal{T}(M; T_r;J;S_*^{+}) = \bigcup_{n=N}^{\infty} \mathcal{T}(M; T_r;J;S_*^{n,+}) \]
where $S_*^{n,+}$ is $S_*$ with the $j$-th factor replaced by $\tau^{-1}([b, n])$.  We similarly define $\mathcal{T}(M; T_r;J;S_*^{-})$ by replacing the $j$-th factor by $\tau^{-1}([-n, a])$.  From the arguments of the case $v(S_*) =0$ we know that $\mathcal{T}(M; T_r;J;S_*^{n, \pm})$ are closed intervals, which are obviously nested.  Thus $\mathcal{T}(M; T_r;J;S_*^{\pm})$ are both intervals, and the reader will verify that 
\[ \mathcal{T}(M; T_r;J;S_*^h) = \mathcal{T}(M; T_r;J;S_*^{+}) \cup \mathcal{T}(M; T_r;J;S_*^{-}) = ( -\infty, c_{max}] \cup [c_{min}, \infty)
\]
for some $c_{max}, c_{min}$ determined by choices of $n$ for which (\ref{cmax}) and (\ref{cmin}) are maximal and minimal.
\end{proof}

\begin{corollary}  \label{cor: relative detection}
$\mathcal{D}(M; T_r; J; S_*)$ is a closed, connected subset of the circle $\mathcal{S}(T_r)$ whose frontier consists of rational slopes. If $\mathcal{D}(M; T_r; S_*)$ is a non-degenerate interval, its endpoints are not contained in $\mathcal{D}_{str}(M; T_r; S_*)$.
\qed 
\end{corollary}

\section{Horizontal foliations in Seifert manifolds} \label{s: horizontal foliations}

In this section $M$ will be a Seifert fibered manifold with base orbifold $\mathcal{B} = P(a_1, \ldots, a_n)$ as in \S \ref{pieces}. Denote by $X \to \mathcal{B}$ a universal cover. Then $\hbox{int}(X) \cong \mathbb R^2$ and $\pi_1(M)$ acts discontinuously on $X$ via the quotient homomorphism $\varphi: \pi_1(M) \to \pi_1(M)/\langle h \rangle  \equiv \pi_1(\mathcal{B})$. 

An important family of subgroups of $\widetilde{\hbox{Homeo}}_+(S^1)$ correspond to the universal covers $\widetilde{PSL}(2,\mathbb R)_k$ of the $k$-fold cyclic covers $PSL(2, \mathbb R)_k$ of $PSL(2, \mathbb R)$ ($k \geq 1$). These groups are conjugate in $\hbox{Homeo}_+(\mathbb R)$, though not in $\widetilde{\hbox{Homeo}}_+(S^1)$. More precisely, let $F_k: \mathbb R \to \mathbb R$ be the homeomorphism $F_k(x) = kx$. Then 
$$\widetilde{PSL}(2,\mathbb R)_k = F_k^{-1} \widetilde{PSL}(2,\mathbb R)F_k$$
Note that $\widetilde{PSL}(2,\mathbb R)_1 = \widetilde{PSL}(2,\mathbb R)$. The composition of conjugation with the cover $\widetilde{PSL}(2,\mathbb R) \to PSL(2; \mathbb R)$ is a universal covering homomorphism identifiable with the composition $\widetilde{PSL}(2,\mathbb R)_k \stackrel{}{\longrightarrow} PSL(2,\mathbb R)_k \to PSL(2; \mathbb R)$. 

An element of $\widetilde{PSL}(2,\mathbb R)_k$ is either {\it elliptic, parabolic}, or {\it hyperbolic} depending on whether its image in $PSL(2, \mathbb R)$ has that property. Thus an element is elliptic if and only if it is conjugate to a translation. The parabolic and hyperbolic elements of $\widetilde{PSL}(2,\mathbb R)$ have integral translation numbers, so the translation number of a parabolic or hyperbolic element of $\widetilde{PSL}(2,\mathbb R)_k$ is of the form $\frac{d}{k}$ where $d \in \mathbb Z$. The image in $PSL(2,\mathbb R)$ of an abelian subgroup of $\widetilde{PSL}(2,\mathbb R)_k$ consists entirely of elliptics, or entirely of hyperbolics, or entirely of parabolics. 

Given a homomorphism $\rho : \pi_1(M) \rightarrow \widetilde{PSL}(2, \mathbb{R})_k$ which satisfies $\rho(h) = \hbox{sh}(1)$, $\pi_1(M)$ acts freely properly discontinuously on $X \times \mathbb R$ via
$$\gamma \cdot (x, t) = (\varphi(\gamma)(x), \rho(\gamma)(t))$$
The quotient map $X \times \mathbb R \to X \times_\rho \mathbb R = (X \times \mathbb R)/ \pi_1(M)$ is a universal cover and $(X \times \mathbb R)/ \pi_1(M)$ is canonically identifiable with $M$ in such a way that the lines $\{x\} \times \mathbb R$ project to the Seifert circles of $M$. The planes $X \times \{t\}$ project to a smooth horizontal foliation $\mathcal{F}(\rho)$ on $M$. (See \cite[\S 6.3]{BC} for instance.)

\begin{lemma} \label{l: condition for fibring}
Let $\rho : \pi_1(M) \rightarrow \widetilde{PSL}(2, \mathbb{R})_k$ be a homomorphism such that $\rho(h) = \hbox{sh}(1)$. Then $\mathcal{F}(\rho)$ is a fibration with compact fibre if and only if the composition of $\rho$ with the projection to $PSL(2,\mathbb R)_k$ is a finite cyclic group\footnote{The latter occurs if and only if the further projection to $PSL(2,\mathbb R)$ is a finite cyclic group}.
\end{lemma}

\begin{proof}
For each $t \in \mathbb R$, let $\hbox{Stab}_\rho(t)$ be the stabiliser of $t$ under the action of $\pi_1(M)$ on $\mathbb R$ determined by $\rho$ and $\hbox{Stab}_\rho^*(t)$ the subgroup of $\pi_1(M)$  generated by $\hbox{Stab}_\rho(t)$ and $\langle h \rangle$. Then $\hbox{Stab}_\rho^*(t)$ is the stabilizer of the image of $t$ in the circle under the action of $\pi_1(M)$ determined by the composition of $\rho$ with the cover $\psi: \widetilde{PSL}(2, \mathbb{R})_k \to PSL(2, \mathbb{R})_k$. Clearly there is a homomorphism $\bar \rho: \pi_1(\mathcal{B}) \to PSL(2, \mathbb{R})_k$ such that $\psi \circ \rho = \bar \rho \circ \varphi$. 

 Since $\rho(h) = \hbox{sh}(1)$, $\hbox{Stab}_\rho(t) \cap \langle h \rangle = \{1\}$, and therefore $\varphi|\hbox{Stab}_\rho(t)$ is one-to-one. In particular, $\varphi(\hbox{Stab}_\rho(t))$ is torsion free, so $X/\varphi(\hbox{Stab}_\rho(t))$ is a surface. By construction, $X/\varphi(\hbox{Stab}_\rho(t)) \times \{t\} = (X \times \{t\})/\hbox{Stab}_\rho(t)$ embeds in $M = X \times_\rho \mathbb R$. 

If the image of $\psi \circ \rho$ is a finite cyclic group, then the kernel of this composition, which is contained in $\hbox{Stab}_\rho^*(t)$, is of finite index in $\pi_1(M)$. Thus $X/\varphi(\hbox{Stab}_\rho(t))$ is a compact surface, call it $F$. Further, the image of $\rho$ consists of elliptics, so $\hbox{Stab}_\rho(t)$ is independent of $t$. Thus $\mathcal{F}(\rho)$ is a foliation by horizontal surfaces, each homeomorphic to $F$. It follows that $\mathcal{F}(\rho)$ is a fibration.

Conversely, suppose that $\mathcal{F}(\rho)$ is a fibration with fibre $F$. Then for each $t \in \mathbb R$, $\varphi(\hbox{Stab}_\rho(t))$ is a subgroup of $\pi_1(\mathcal{B})$ whose index $d = \chi(F)/\chi(\mathcal{B})$ is independent of $t$. Thus there are only finitely many possibilities for $\varphi(\hbox{Stab}_\rho(t))$. 

Consider a conjugate of $\varphi(\hbox{Stab}_\rho(t))$ in $\pi_1(\mathcal{B})$, say $x \varphi(\hbox{Stab}_\rho(t)) x^{-1}$ where $x \in \pi_1(\mathcal{B})$. Fix $\xi \in \pi_1(M)$ which projects to $x$. Then $\hbox{Stab}_\rho(\rho(\xi)(t)) = \xi \hbox{Stab}_\rho(t) \xi^{-1}$, and therefore  $\varphi(\hbox{Stab}_\rho(\rho(\xi)(t)))$ $= x \varphi(\hbox{Stab}_\rho(t)) x^{-1}$, so the set of images in $\pi_1(\mathcal{B})$ of stabilisers of points of $\mathbb R$ is invariant under conjugation. As they each have fixed index $d$ in $\pi_1(\mathcal{B})$, their intersection is a finite index normal subgroup of $\pi_1(\mathcal{B})$. Clearly this intersection is the kernel of $\rho$. But then if $\xi \in \hbox{Stab}_\rho(t)$, there is some $n \geq 1$ such that $\xi^n \in \hbox{ker}(\rho)$. It follows that $\xi \in  \hbox{ker}(\rho)$ (i.e. $\widetilde{PSL}(2, \mathbb{R})_k$ is torsion free). Thus $\hbox{Stab}_\rho(t) = \hbox{ker}(\rho)$ for all $t$. In particular, $\varphi(\hbox{ker}(\rho))$ has index $d$ in $\pi_1(\mathcal{B})$. But since $(\bar \rho \circ \varphi)(\hbox{ker}(\rho)) =  (\psi \circ \rho)(\hbox{ker}(\rho)) = \{1\}$, the image of $\bar \rho$ is a finite subgroup of $PSL(2,\mathbb R)_k$, and therefore finite cyclic.
\end{proof}

\begin{lemma} \label{l:multiple boundary components}
Suppose that $\rho: \pi_1(M) \rightarrow \widetilde{PSL}(2, \mathbb{R})_k$ is a representation satisfying $\rho(h) = \hbox{sh}(1)$ such that $\rho$ has nonabelian image or there is some $\gamma \in \pi_1(M)$ such that $\rho(\gamma)$ has irrational translation number. Given $T \subset \partial M$, if $\rho|{\pi_1(T)}$ consists of translations by rational numbers, then the number of components of $L \cap T$ is infinite for the generic leaf $L$ of $\mathcal{F}(\rho)$.
\end{lemma}

\begin{proof}
First suppose that $\rho$ has nonabelian image. A subgroup of $\widetilde{PSL}(2, \mathbb{R})_k$ consisting of elliptics is conjugate in $\widetilde{PSL}(2, \mathbb{R})_k$ into the group of translations. (See \cite[Theorem 4.3.7]{Bn}.) Thus as the image of $\rho$ is non-abelian, there is a $\gamma \in \pi_1(M)$ such that $\rho(\gamma)$ is parabolic or hyperbolic. Denote by $f$ the associated element of $\hbox{Diff}_+(\mathbb R)$. Then $f$ has finitely many fixed points in $\mathbb R / \mathbb Z$ and if $t \in \mathbb R$ is not such a fixed point, the sequence $\{f^j(t)\}$ converges in $\mathbb R / \mathbb Z$ to one of them.

Fix a boundary component $K_0$ of $X$ whose image in $\mathcal{B}$ is contained in the image of $T$ under the quotient $M \to \mathcal{B}$. For each $t \in \mathbb R$, the image of $K_0 \times \{t\}$ is contained in $T$ and is a boundary component of a leaf of $\mathcal{F}(\rho)$. Hence it is a circle $C_t$. There is an integer $m \geq 1$ (depending on the slope of $C_t$ on $T$) such that $C_t = C_{t'}$ if and only if $m(t - t') \equiv 0$ (mod $1$).

For $j \in \mathbb Z$ let $K_j = \varphi(\gamma)^{-j}(K_0)$ be the image of $K_0$ under the action by deck transformations of $\pi_1(\mathcal{B}) = \pi_1(M)/\langle h \rangle$ on $X$. For each $t \in \mathbb R$ and $j \in \mathbb Z$, $K_j \times \{t\}$ is contained in $X \times \{t\}$ and so their images in $M$ are contained in a leaf $L_t$ of $\mathcal{F}(\rho)$. On the other hand, since $\gamma^j(K_j \times \{t\}) = K_0 \times \{f^j(t)\}$, $L_t$ contains $C_{f^j(t)}$ for all $j$. We noted above that for generic $t$, the sequence $\{f^j(t)\}$ converges in $\mathbb R / \mathbb Z$. We also noted that $C_t = C_{t'}$ if and only if $m(t - t') \equiv 0$ (mod $1$). It follows that for generic $t$, the number of components of $L_t \cap T$ is infinite.

A similar proof works when there is some $\gamma \in \pi_1(M)$ such that $f = \rho(\gamma)$ has irrational translation number. For in this case, the sequence $\{f^j(t)\}$ is dense in $\mathbb R / \mathbb Z$ for each $t \in \mathbb R$. 
\end{proof}

\section{Adding holonomy} \label{sec: operations}

Throughout this section $M$ will denote a piece of graph manifold rational homology solid torus (cf. \S \ref{pieces}) with boundary components $T_1, \ldots, T_r$ and base orbifold $\mathcal{B}$. We use $X \to \mathcal{B}$ to denote a universal cover. 

\subsection{Interval hyperbolic foliations} 

For each positive integer $k$ let $IH^\infty(k)$ be the set of orientation-preserving diffeomorphisms $f: S^1 \to S^1$ whose fixed point set consists of $2k$ disjoint closed non-degenerate intervals and for which $f$ is alternately increasing or decreasing on the $2k$ complementary open intervals.  

For a positive integer $k$ we say that a codimension one foliation $\mathcal{F}$ on a torus $T$ is (smoothly) {\it $k$ interval-hyperbolic} if it is diffeomorphic to the suspension foliation $\mathcal{F}(f)$ on $T(f)$ of a homeomorphism $f \in IH^\infty(k)$. 

The goal of this section is to prove the following proposition whose proof follows immediately from Propositions \ref{elliptic to ih}, \ref{compact 2}, and \ref{compact 3}. 

\begin{proposition} \label{fixing the pieces}
Let $M$ be a Seifert manifold as in \S \ref{pieces} which admits a smooth, co-oriented taut foliation $\mathcal{F}$. Suppose, additionally, that if $\mathcal{F}$ has a compact horizontal leaf $F$ then $|\partial M| +  |\partial F| \geq 3$. Then given an odd integer $k \gg 0$, there is a smooth co-oriented taut foliation $\mathcal{F}'$ on $M$ which detects $[\alpha_*(\mathcal{F})]$ such that for $1 \leq i \leq r$ the foliation $\mathcal{F}'|T_{i}$ is $k$-interval hyperbolic whenever $[\alpha_i(\mathcal{F})]$ is rational.  
\qed
\end{proposition}

\subsection{Adding holonomy in the horizontal non-fibred case} \label{non-fibred case}
Suppose that $[\alpha_*]$ is a horizontal element of $\mathcal{S}(M)$ and $\rho: \pi_1(M) \to \widetilde{PSL}(2, \mathbb{R})_k$ such that $\rho(h) = \hbox{sh}(1)$ and $[\alpha_*]$ is $\mathcal{F}(\rho)$-detected. (So the base orbifold of $M$ is orientable by \cite[Proposition 6.5]{BC}.) According to \cite[Proposition A.8]{BC}, we can assume that $\mathcal{F}(\rho)$ is either elliptic or hyperbolic on each component of $\partial M$. Assume as well that $\mathcal{F}(\rho)$ is not a fibration. According to Lemma \ref{l: condition for fibring}, the image of $\rho$ is either 
\vspace{-.2cm} 
\begin{itemize}

\item non-abelian, or 

\vspace{.2cm} \item consists of elliptics at least one of which has irrational translation number, or 

\vspace{.2cm} \item consists of hyperbolics and elements of the centre of $\widetilde{PSL}(2, \mathbb{R})_k$. 

\end{itemize}

\begin{proposition} \label{elliptic to ih}
Let $M$ be a Seifert manifold with non-empty boundary and suppose that $\rho : \pi_1(M) \rightarrow \widetilde{PSL}(2, \mathbb{R})_j$ is as above and that $\mathcal{F}(\rho)$ is not a fibration. Then for each odd $k \geq 1$ and every sequence $1 \leq i_1 < i_2 < \ldots < i_s \leq r$ such that $\mathcal{F}|{T_{i_l}}$ is elliptic of rational slope for $1 \leq l \leq s$, there exists a smooth, co-oriented, horizontal  foliation $\mathcal{F}'$ on $M$ such that $[\alpha_*(\mathcal{F}')] = [\alpha_*(\mathcal{F}(\rho))]$, $\mathcal{F}'|{T_{i_l}}$ is $k$-interval hyperbolic for $1 \leq l \leq s$, and $\mathcal{F}'|{T_{i}} = \mathcal{F}|{T_{i}}$ for $i \not \in \{i_1, i_2, \ldots, i_s\}$. 
\end{proposition}

\begin{proof}
Without loss of generality, we may suppose that $i_l = l$ for each $l$. 

From above we know that the image of $\rho$ is either non-abelian or consists of elliptics at least one of which has irrational translation number. Hence Lemma \ref{l:multiple boundary components} implies that for $1 \leq i \leq s$, the number of components of $L \cap T_i$ is infinite for the generic leaf $L$ of $\mathcal{F}(\rho)$. Fix disjoint leaves $L_1, L_2, \ldots , L_{s}$ such that for $1 \leq l \leq s$, $L_l \cap T_{l}$ consists of infinitely many circles of slope $[\alpha_{l}(\mathcal{F}(\rho))]$.

Consider a smooth arc $A$ properly embedded in $L_1$ which runs between distinct  components of $L_1 \cap T_1$. Since $\mathcal{F}(\rho)$ is horizontal, $A$ can be extended to a smooth vertical surface $A \times I$, 
%contained in an arbitrarily small neighbourhood of $L_1$ in $M \setminus \bigcup_{i= 2}^s L_i$, 
where $A = A \times \{\frac12\}$, $A \times \{t\}$ is contained in a leaf of $\mathcal{F}(\rho)$ for each $t \in I$, and $\{x\} \times I$ is contained in a Seifert fibre of $M$ for each $x \in A$. We can assume, moreover, that there are disjoint annular regions $B_0, B_1$ of $T_1$, each foliated by circles of $\mathcal{F}(\rho) \cap T_1$ and each containing one of the components $a$ of $(A \times I) \cap T_1$ as a properly embedded arc.

Fix an orientation-preserving diffeomorphism $h$ of the unit interval $I$ which is strictly increasing on $(0,1)$ and 
$$h^{(n)}(0) = h^{(n)}(1) = \left\{ \begin{array}{ll} 1 & \hbox{ if } n = 1 \\ 0 & \hbox{ if } n > 1 \end{array} \right.$$
Let $f$ be one of the diffeomorphisms $h$ or $h^{-1}$ and cut $M$ open along $A \times I$ then reglue according to the rule $(x,t) \mapsto (f(x), t)$ for $(x, t) \in A \times I$ (cf.~\cite[Operation 2.2]{Ga}). This operation transforms $\mathcal{F}(\rho)$ into a new (smoothable) horizontal foliation $\mathcal{F}_1$ on $M$ which can be assumed to differ from $\mathcal{F}(\rho)$ only in an arbitrarily small neighbourhood of $A \times I$. Further, its restrictions to $B_0$ and $B_1$ are suspensions of $h$ or $h^{-1}$. By performing this operation on an appropriately chosen finite family of disjoint, smooth, properly embedded arcs in $L_1$, we can produce a smooth foliation $\mathcal{F}_1$ on $M$ such that $[\alpha_*(\mathcal{F}_1)] = [\alpha_*(\mathcal{F}(\rho))]$, $\mathcal{F}'|{T_1}$ is $k$-interval hyperbolic, and $\mathcal{F}'|{T_{i}} = \mathcal{F}|{T_{i}}$ for $2 \leq i \leq r$. Choosing appropriate arcs in $L_2, \ldots, L_s$ and performing similar modifications to $\mathcal{F}$ in disjoint neighbourhoods of those arcs completes the proof. 
\end{proof}

\begin{addendum} \label{addendum} 
If, in the proposition above, $\mathcal{F}(T_{j_l})$ is $c_l$-hyperbolic for $1 \leq j_1 < j_2 < \ldots < j_t \leq r$, then for each $c'_i \geq c_i$ we can produce a co-oriented horizontal foliation $\mathcal{F}'$ on $M$ such that $[\alpha_*(\mathcal{F}')] = [\alpha_*(\mathcal{F}(\rho))]$, $\mathcal{F}'|{T_{i_l}}$ is $k_l$-interval hyperbolic for $1 \leq l \leq s$, $\mathcal{F'}|{T_{j_l}}$ is $c_l'$-interval hyperbolic for $1 \leq l \leq t$, and $\mathcal{F}'|{T_{i}} = \mathcal{F}|{T_{i}}$ otherwise. Such a foliation cannot be assumed smooth since the proof involves thickening leaves, which does not preserve smoothness (cf. \cite[Lemma 6.10]{BC}). 
\end{addendum}

\subsection{Adding holonomy in the compact leaf case}

In this section we prove Proposition \ref{fixing the pieces} in the case that $M$ admits a smooth foliation with a compact leaf. See Propositions \ref{compact 2} and \ref{compact 3}. First though, we show how to obtain foliations with hyperbolic behaviour on the boundary of $M$. 

\begin{proposition} \label{compact} 
Let $M$ be a Seifert manifold as in \S \ref{pieces} which admits a smooth, co-oriented taut foliation $\mathcal{F}$ with a compact leaf $F$ such that $|\partial M| + |\partial F| \geq 3$. Then there is a smooth co-oriented taut foliation on $M$ which detects $[\alpha_*(\mathcal{F})]$ and is hyperbolic on each component of $\partial M$. 
\end{proposition}

\begin{proof}
First assume that $F$ is horizontal in $M$. Then $M$ fibres over the circle with fibre $F$ and the manifold obtained by cutting $M$ open along $n \geq 1$ fibres has $n$ components, each diffeomorphic to a product $F \times I$. Further, a boundary component $T_j$ of $M$ is cut into $n|\partial F \cap T_j|$ annuli. 

We can create holonomy in any of the product regions $F \times I$ as follows (cf.~\cite[Operation 2.2]{Ga}). Given a smooth, properly embedded arc $A$ in $F$ running between distinct boundary components $C, C'$ of $F$ and an orientation-preserving diffeomorphism $f: I \to I$, cut $F \times I$ open along $A \times I$ then reglue according to the rule $(x,t) \mapsto (f(x), t)$ for $(x, t) \in A \times I$. The product foliation on $F \times I$ is transformed into a new smoothable foliation of $M$ which coincides with the old one near $\big(\partial F \setminus (C \cup C')\big) \times I$. On the other hand, the induced foliations on $C \times I$ and $C' \times I$ are diffeomorphic to the suspension of $f$. Choosing a finite number of arcs whose union contains points of each boundary component of $F$, we can produce a smooth, co-oriented foliation on $F \times I$ which restricts to the suspension foliation of an increasing or decreasing diffeomeorphism of $I$ on each component of $\partial F \times I$. Note that $F \times \{0\}$ and $F \times \{1\}$ are leaves of the new foliation. After an appropriate choice of $n$, arcs $A$ and diffeomorphisms $f$ (either strictly increasing or strictly decreasing) for each of the $n$ copies of $F \times I$, we can produce a smooth co-oriented taut foliation on $M$ which detects $[\alpha_*(\mathcal{F})]$ and which is hyperbolic on all of the boundary components of $M$. 

Next assume that $F$ is vertical, and therefore an annulus. Without loss of generality we can suppose that $[\alpha_i(\mathcal{F})]$ is vertical if and only if $1 \leq i \leq s$. Cut $M$ open along a finite number of vertical annuli as in the proof of \cite[Proposition 6.5]{BC} and fix a fibration on the resulting manifold $M'$ which detects $[\alpha_i]$ on $T_i$ for $s + 1 \leq i \leq r$. Add hyperbolic behaviour on the boundary components of $M'$ as in the previous paragraph. Finally we reglue and spin around the vertical annuli to produce hyperbolic behaviour on the components of $\partial M$. 
\end{proof}

\begin{proposition} \label{compact 2} 
Let $M$ be a Seifert manifold as in \S \ref{pieces} which admits a smooth, co-oriented taut foliation $\mathcal{F}$ with a compact leaf $F$ such that $|\partial M| + |\partial F| \geq 3$. For each odd $k \geq 1$, there is a smooth co-oriented taut foliation on $M$ which detects $[\alpha_*(\mathcal{F})]$ and is $k$-interval hyperbolic on each component of $\partial M$. 
\end{proposition}

\begin{proof}
Proceed as in the proof of Proposition \ref{compact}, though when we construct the foliations on the components $F \times I$ we choose arcs which are disjoint from certain boundary components of $F$, thus leaving leaving the product foliation on these components of $\partial F \times I$ alone. The reader will verify that if $|\partial M| + |\partial F| \geq 4$, then for each integer $k \geq 1$ we can choose $n$, the arcs, and the diffeomorphisms $f$ (as in the proof of Proposition \ref{elliptic to ih}) to create the appropriate interval hyperbolic behaviour on $T_1, \ldots, T_r$ while leaving the remaining boundary components alone. On the other hand, if $|\partial M| + |\partial F| = 3$, so $s = 1$ and $|\partial F \cap T_1| = 2$, we can only construct $k$-interval hyperbolic behaviour on the boundary of $T_1$ for odd $k$, owing to the relative position on $\partial M$ of the boundary components of a finite number of fibres of the fibration. 
\end{proof}

\begin{proposition} \label{compact 3} 
Let $M$ be a Seifert manifold as in \S \ref{pieces} and $[\alpha_*] \in \mathcal{S}(M)$ such that $v([\alpha_*]) > 0$ and is even if $M$ has orientable base orbifold. Without loss of generality suppose that $[\alpha_i]$ is rational if and only if $1 \leq i \leq t$ and $[\alpha_i] = [h]$ if and only if $1 \leq i \leq v = v([\alpha_*])$. Fix an odd integer $k \geq 1$. Then there is a smooth foliation $\mathcal{F}$ on $M$ which detects $[\alpha_*]$ and which is $k$ or $k+2$-interval hyperbolic on $T_1$ and $k$-interval hyperbolic on $T_i$ for $2 \leq i \leq t$. 
\end{proposition}

\begin{proof} 
Let $M_1 = M$ if $M$ has orientable base orbifold. Otherwise split $M$ along a vertical annulus $A_0$ running from $T_1$ to itself to produce a Seifert manifold $M_1$ with planar base orbifold. Then $\partial M_1 = T_1' \cup T_2 \cup \ldots \cup T_r$ where $T_1' = T_1$ when $M$ has orientable base orbifold. When $M$ has a non-orientable base orbifold, $T_1'$ contains disjoint essential annuli $A_0^+, A_0^-$ corresponding to $A_0$.  

Fix a homomorphism $\rho: \pi_1(M_1) \to \widetilde{PSL}(2, \mathbb R)$ such that the image of $\rho$ consists of translations, $\rho(h) = \hbox{sh}(1)$, $[\alpha_i(\mathcal{F}(\rho))]  = [\alpha_i]$ for $s+1 \leq i \leq r$, and $[\alpha_i(\mathcal{F}(\rho))]$ is rational for $3 \leq i \leq t$. If $[\alpha_*]$ is not rational this may require that $[\alpha_1(\mathcal{F}(\rho))]$ be irrational. If $[\alpha_*]$ is rational and $v \geq 2$, we can suppose that $[\alpha_1(\mathcal{F}(\rho))]$ and $[\alpha_2(\mathcal{F}(\rho))]$ are irrational. 

{\bf Case 1}. $[\alpha_*]$ is not rational or $v \geq 2$. 

Under these conditions, $\mathcal{F}(\rho)$ is not a fibration, so Proposition \ref{elliptic to ih} implies that there exists a smooth, co-oriented, horizontal  foliation $\mathcal{F}_1$ on $M_1$ such that $[\alpha_*(\mathcal{F}_1)] = [\alpha_*(\mathcal{F}(\rho))]$, $\mathcal{F}_1|{T_{i}} = \mathcal{F}|{T_{i}}$ for $1 \leq i \leq v$, and $\mathcal{F}_1|{T_{i}}$ is $k$-interval hyperbolic for $v + 1 \leq i \leq t$. 

Suppose that $M$ has orientable base orbifold. Then $M_1 = M$ and $v \geq 2$ is even. To build $\mathcal{F}$, spin $\mathcal{F}_1$ vertically around an appropriately chosen family of disjoint, properly-embedded, vertical annuli running between $T_{2i-1}$ and $T_{2i}$ for each $i = 1,  \ldots, v/2$ (cf. \cite[Lemma 6.5]{BC}) so as to obtain the desired interval hyperbolic behaviour on $T_1, \ldots, T_v$. This operation leaves $\mathcal{F}_i$ alone near $T_i$ for $v+1 \leq i \leq r$. 

Suppose next that $M$ has non-orientable base orbifold and $v$ is even. Choose a family of disjoint, properly-embedded, vertical annuli running between $T_{2i-1}$ and $T_{2i}$ in $M_1$ for $2 \leq i \leq v/2$ and spin about them to fix things on $T_3, \ldots, T_u$. Next connect $T_1$ and $T_2$ by two disjoint, properly-embedded, vertical annuli  $A_+, A_-$, disjoint the previous chosen ones, such that the union of $A_+, A_-$ and an annulus in $T_2$ is (non-ambiently) isotopic to $A_0^+$. Now spin appropriately about $k$ parallel copies of $A_+$ and $k$ parallel copies of $A_-$ and reglue $A_0^+$ to $A_0^-$ to reform $M$ in order to obtain $k$-interval hyperbolic behaviour on $T_1$ and $T_2$. 

Finally suppose that $M$ has non-orientable base orbifold and $v$ is odd. Choose an appropriate family of disjoint, properly-embedded, vertical annuli running between $T_{2i}$ and $T_{2i+1}$ in $M_1$ for $1 \leq i \leq (v-1)/2$ and spin vertically around them to fix things on $T_2, \ldots, T_v$. Now reglue $A_0^+$ to $A_0^-$ to reform $M$ and spin about $k$ parallel copies of $A_0$ to fix things on $T_1$. 

{\bf Case 2}.  $[\alpha_*]$ is rational and $v = 1$. 

Since $v = 1$, $M$ has non-orientable base orbifold, and since $[\alpha_*]$ is rational, $\mathcal{F}(\rho)$ is a fibration on $M_1$ (Lemma \ref{l: condition for fibring}).

If $r = 1$, reglue $A_0^+$ to $A_0^-$ to reform $M$ and spin about $k$ parallel copies of $A_0$ to arrange for $k$-interval hyperbolic behaviour on $T_1$. 

Next suppose that $r > 1$ and let $F$ be a fibre of the fibration $\mathcal{F}(\rho)$. If $(r-1) + \sum_{i = 2}^r |\partial F \cap T_i| \geq 3$ we can spin about $k$ parallel copies of $A_0$ to fix things on $T_1$ and then use the argument of the proof of Proposition \ref{compact 2} applied to $\mathcal{F}(\rho)$ and the tori $T_2, \ldots, T_r$ to complete the proof.  

Finally suppose that $r > 1$ and $(r-1) + \sum_{i = 2}^r |\partial F \cap T_i| \leq 2$. Then $r = 2$ and $|\partial F \cap T_2| = 1$. Use the same argument to arrange $k$-interval hyperbolic behavior on $T_1'$ and $T_2$. Then reglue $A_0^+$ to $A_0^-$ to reform $M$ and spin about $k$ parallel copies of $A_0$ to fix things on $T_1$.  
\end{proof}

\section{Slope detection for graph manifold rational homology solid tori} \label{sec: slope detection graph}

In this section we  derive characterisations of slope detection (Theorem \ref{theorem: gluing}) and strong slope detection (Theorem \ref{theorem: gluing 2}) in graph manifold rational homology solid tori. These are used in \S \ref{subsec: slope detection graph} to study the space of detected slopes of a graph manifold rational homology solid torus.   

Throughout we fix a graph manifold rational homology solid torus $V$ with JSJ pieces 
$M_1, \ldots, M_n$ where $M_1$ is the piece containing $\partial V$. 
Let $\mathcal{T}(V)= \{T_0 = \partial V, T_1, T_2, \ldots ,T_m\}$ where $T_1, \ldots, T_m$ are the JSJ tori of $V$. Then 
$\mathcal{S}(\mathcal{T}(V)) = \mathcal{S}(V) \times \mathcal{S}(T_1) \times \ldots \times \mathcal{S}(T_m)$ 
projects onto $\mathcal{S}(M_i)$ for each $i$. 
For $[\alpha_*] \in \mathcal{S}(\mathcal{T}(V))$ we use
$$[\alpha_*^{(i)}] \in \mathcal{S}(M_i)$$ 
to denote this projection and call $[\alpha_*]$ {\it gluing coherent} if $[\alpha_*^{(i)}]$ is detected for all $i$. 

\subsection{A characterisation of detected slopes on the boundary of graph manifold rational homology solid tori}
We take the convention in this and the next subsection that the parenthetical phrases in the statements of results are to be either simultaneously considered or simultaneously ignored.

\begin{theorem} \label{theorem: gluing} 
Let $V$ be a graph manifold rational homology solid torus and fix $[\alpha] \in \mathcal{S}(V)$. Then $V$ admits a $($horizontal$)$ co-oriented taut foliation which detects $[\alpha]$ if and only if $[\alpha]$ extends to a $($horizontal$)$ gluing coherent element $[\beta_*] \in \mathcal{S}(\mathcal{T}(V))$. 
\end{theorem}

\begin{proof}
Without loss of generality we can suppose that $n > 1$. 

 The forward implication of (1)  is straightforward; a (horizontal) co-oriented taut foliation $\mathcal{F}$ on $V$ can be isotoped so that it is transverse to $\mathcal{T}(V)$ and intersects each $M_i$ in a (horizontal) co-oriented taut foliation. Hence it determines a (horizontal) gluing coherent element $[\beta_*] \in \mathcal{S}(V; \mathcal{T})$ which extends $[\alpha]$.  

Conversely suppose that $[\alpha]$ extends to a gluing coherent $($horizontal$)$ element $[\beta_*] \in \mathcal{S}(\mathcal{T}(V))$. Our first task is to modify $[\beta_*]$ (rel $[\alpha]$) so that we can apply the results of \S \ref{sec: operations}. 

Suppose that $v([\beta_*^{(i)}])$ is odd for some $i$ such that $M_i$ has an orientable base orbifold. In this case $v([\beta_*^{(i)}]) \geq 3$ by \cite[Proposition 6.5]{BC}. Without loss of generality we can suppose that $\partial M_i = T \cup T_1 \cup \ldots \cup T_{r}$ where $T$ is either $\partial V$ or the JSJ torus in $\partial M_i$ which separates $M_i$ from $\partial V$. Then $V = U \cup V_1 \cup \ldots \cup V_r$ where
\vspace{-.3cm} 
\begin{itemize}
\item $U$ is empty if $T = \partial V$ and is the connected graph submanifold of $V$ with boundary $\partial V \cup T$ otherwise; 

\vspace{.2cm}\item $V_1, \ldots, V_{r}$ are the graph manifold rational homology solid tori in $V$ such that $\partial V_j = T_j$.
\end{itemize} 
\vspace{-.3cm}
We can suppose that $i$ is chosen so that $U$ contains no pieces $M_j$ with $v([\beta_*^{(j)}])$ odd and orientable base orbifold.  Since $V$ is a rational homology solid torus, there is at most one $j \in \{1, \ldots, r\}$ such that $[\beta_j]$ is the rational longitude of $V_j$. Thus, as $r \geq v([\beta_*^{(i)}]) - 1 \geq 2$, 
there is a $j$ such that $[\beta_j] \ne [\lambda_{V_j}]$. Then by \cite[Lemma 10.3]{BC}, there is another gluing coherent family of (horizontal) slopes $[\beta_*']$ which extends $[\alpha]$, agrees with $[\beta_*]$ on the JSJ tori of $V$ contained in $V \setminus V_j$, and for which $v([(\beta_*')^{(i)}]) = v([\beta_*^{(i)}]) - 1$. Since the graph of the $V$ is a tree rooted at its outermost piece, it is easy to see that we can continue to make such modifications without affecting those already made, and thus arrive at an extension of $[\alpha]$, which we continue to denote by $[\beta_*]$, such that $v([\beta_*^{(i)}])$ is even whenever it is positive and $M_i$ has an orientable base orbifold. By the proof of \cite[Lemma 10.2]{BC}, we can also assume that the components of $[\beta_*]$ corresponding to the JSJ tori of $V$ are rational. We can now use Proposition \ref{fixing the pieces} to choose co-oriented taut foliations on the pieces of $V$ which piece together (cf. Lemma 6.9 of \cite{BC}) to form a co-oriented taut foliation on $V$ which detects $[\alpha]$. This completes the proof. 
\end{proof}

\begin{corollary} \label{cor: detection}
Suppose that $\partial M_1 = \partial V \cup T_1 \cup \ldots \cup T_r$ where $r \geq 0$ and $V_j$ is the graph manifold rational homology solid torus in $V$ bounded by $T_j$ $(1 \leq j \leq r)$. Set $D_j = \mathcal{D}(V_j)$ and $D_* = D_1 \times \ldots \times D_{r}$. Then 

$(1)$ $\mathcal{D}(V) = \mathcal{D}(M_1; \partial V; \emptyset; D_*)$.

$(2)$ $\mathcal{D}(V)$ is a closed, connected subset of the circle $\mathcal{S}(\partial V)$ whose frontier consists of rational slopes. 

$(3)$ $[\lambda_V] \in \mathcal{D}(V)$. 
\end{corollary}

\begin{proof} 
We proceed by induction on the number $n$ of pieces of $V$. If $n = 0$ the identity $\mathcal{D}(V) = \mathcal{D}(M_1; \partial V; \emptyset; D_*)$ is obvious while assertion (2)  follows from Corollary \ref{cor: relative detection}. Finally assertion (3) is a consequence of \cite[Proposition 6.5 and Corollary A.7]{BC}. 

Suupose that $n > 0$. Theorem \ref{theorem: gluing} implies that $\mathcal{D}(V) = \mathcal{D}(M_1; \partial V; \emptyset; D_*)$. By induction, each $D_j$ is a closed, connected subset of the circle $\mathcal{S}(\partial V)$ which contains $[\lambda_{V_j}]$ and whose frontier consists of rational slopes. Corollary \ref{cor: relative detection} then implies that 
$\mathcal{D}(V) = \mathcal{D}(M_1; \partial V; \emptyset; D_*)$ is a closed, connected subset of the circle $\mathcal{S}(\partial V)$ whose frontier consists of rational slopes. To complete the induction we must show that $[\lambda_V] \in \mathcal{D}(V)$. 

Let $F$ be a compact, connected, orientable, incompressible surface in $V$ whose oriented boundary consists of like-oriented curves of slope $[\lambda_V]$. Assume, moreover, that $F$ is chosen to intersect the JSJ tori of $V$ minimally amongst all such surfaces. Then for each $j$, $F \cap \partial V_j$ is either empty or consists of curves of slope $[\lambda_{V_j}]$,  while $F \cap M$ is an essential surface in $M_1$, and so is either horizontal or vertical. In the former case, $M_1$ fibres over the circle with fibre $F \cap M_1$, so $([\lambda_V], [\lambda_{V_1}], \ldots, [\lambda_{V_{r}}])$ is foliation detected in $M_1$. On the other hand, if $F \cap M_1$ is vertical, it is an annulus and \cite[Proposition 6.5]{BC} implies that $([\lambda_V], [\lambda_{V_1}], \ldots, [\lambda_{V_{r}}])$ is foliation detected in $M_1$. In either case, it follows from Theorem \ref{theorem: gluing} and our inductive hypothesis that $[\lambda_V]$ is foliation detected in $V$.      
\end{proof}

\subsection{A characterisation of strongly detected slopes on the boundary of graph manifold rational homology solid tori}

\begin{theorem} \label{theorem: gluing 2} 
Let $V$ be a graph manifold rational homology solid torus and fix $[\alpha] \in \mathcal{S}(V)$. Assume that if $[\alpha]$ is rational and horizontal and $M_1([\alpha])$ fibres over the circle with fibre $F$ such that $|\partial F| = 1$ and $[\partial F]$ is not strongly detected in $V_1 = \overline{V \setminus M_1}$, then neither of the following occurs 
\vspace{-.3cm} 
\begin{itemize}
\item $M_1$ is a cable space and $\Delta([\alpha], h) = 1$, so $M_1([\alpha]) \cong S^1 \times D^2$;

\item $M_1$ has base orbifold an annulus with cone points and $\mathcal{D}(V_1) = \{[\partial F]\}$, so $[\alpha] = [\lambda_V]$. 

\end{itemize}
\vspace{-.3cm} 
Then $V$ admits a $($horizontal$)$ co-oriented taut foliation which strongly detects $[\alpha]$ if and only if $[\alpha]$ extends to a $($horizontal$)$ gluing coherent element $[\beta_*]$ such that $[\beta_*^{(1)}] \in \mathcal{D}_{str}(M_1; \partial V)$.  
\end{theorem}

\begin{remark}
{\rm Suppose that $[\alpha]$ is rational and horizontal and $M_1([\alpha])$ fibres over the circle with fibre $F$ such that $|\partial F| = 1$ and $[\partial F]$ is not strongly detected in $V_1 = \overline{V \setminus M_1}$. Since the only co-oriented taut foliation on $S^1 \times D^2$ is the fibration by disks, if $M_1([\alpha]) \cong S^1 \times D^2$, then $[\alpha]$ is not strongly detected in $V$ even if it extends to a $($horizontal$)$ gluing coherent element $[\beta_*]$ such that $[\beta_*^{(1)}] \in \mathcal{D}_{str}(M_1; \partial V; J; [\alpha_*])$. Thus these cases must be excluded from the theorem. (This corresponds to the notion of ``gluing obstructed" in \cite{BC}.) 
On the other hand, it is not clear to us exactly when the exclusion of the case that $M_1$ has base orbifold an annulus with cone points, $\mathcal{D}(V_1) = \{[\partial F]\}$ and $\mathcal{D}_{str}(V_1) = \emptyset$ is necessary. The problem is to understand when the rational longitude of $M_1([\alpha])$ is detected by a smooth foliation with hyperbolic behavior on $\partial M_1([\alpha])$. This is not always possible when we restrict to transversely projective foliations.}
\end{remark}

\begin{proof}[Proof of Theorem \ref{theorem: gluing 2}]
Without loss of generality we can suppose that $n > 1$. 

As in the proof of Theorem \ref{theorem: gluing}, a (horizontal) co-oriented taut foliation $\mathcal{F}$ on $V$ which strongly detects $[\alpha]$ determines a (horizontal) gluing coherent element $[\beta_*]$ which extends $[\alpha]$ such that $[\beta_*^{(1)}] \in \mathcal{D}_{str}(M_1; \partial V)$.   

Consider the converse and fix a (horizontal) gluing coherent element $[\beta_*]$ which extends $[\alpha]$ such that $[\beta_*^{(1)}] \in \mathcal{D}_{str}(M_1; \partial V)$. The proof of Theorem \ref{theorem: gluing} implies that detected irrational slopes are strongly detected. (The foliations which detect a given slope are obtained by modifying certain model foliations on the pieces. The model foliations strongly detect irrational slopes and the modifications do not alter the foliations near boundary components for which the slope is irrational.) Thus we can suppose that $[\alpha]$ is rational. We can also suppose that $[\alpha]$ is horizontal as otherwise $V$ would be the twisted $I$ bundle over the Klein bottle (\cite[Lemma 6.4]{BC}), contradicting the assumption that $n > 1$. Then $M_1([\alpha])$ is Seifert fibred and as $n > 1$, $M_1([\alpha])$ is irreducible. Since $[\beta_*^{(1)}] \in \mathcal{D}_{str}(M_1; \partial V)$, there is a foliation $\mathcal{F}_0$ on $M_1([\alpha])$ which detects the projection of $[\beta_*]$ to $\mathcal{S}(M_1([\alpha]))$. Given the cases for $M_1$ that we have excluded, Proposition \ref{fixing the pieces} implies that for any odd $k \gg 0$ we can suppose that $\mathcal{F}_0$ is $k$-interval hyperbolic on $\partial M_1([\alpha])$. The same result allows us to choose co-oriented taut foliations on the pieces of $\overline{V \setminus M_1}$ which piece together (cf. Lemma 6.9 of \cite{BC}) with $\mathcal{F}_0$ to form a co-oriented taut foliation on $V$ which strongly detects $[\alpha]$. This completes the proof. 
\end{proof}

\begin{corollary} \label{cor: strong detection}
Suppose that $\partial M_1 = \partial V \cup T_1 \cup \ldots \cup T_r$ where $r \geq 0$ and $V_j$ is the graph manifold rational homology solid torus in $V$ bounded by $T_j$ $(1 \leq j \leq r)$. Set $D_j = \mathcal{D}(V_j)$ and $D_* = D_1 \times \ldots \times D_{r-1}$. Then $\mathcal{D}(V) \setminus \mathcal{D}_{str}(V)$ is finite and consists of rational slopes. Further, $\mathcal{D}_{str}(V) = \mathcal{D}_{str}(M_1; \partial V; D_*)$ unless, perhaps, one of the following situations involving $V_1 = \overline{V \setminus M_1}$ arises: 

$(a)$ $M_1$ is a cable space and there is a rational, horizontal slope $[\alpha] \in \mathcal{D}_{str}(M_1; \partial V; D_*)$ such that $M_1([\alpha]) \cong S^1 \times D^2$ and the meridional slope of $M_1([\alpha])$ is not strongly detected in $V_1$.  

$(b)$ $[\alpha] = [\lambda_V]$, $|\partial M_1| = 2$, say $\partial M_1= \partial V \cup T$, $M_1$ fibres over the circle with fibre $F$ such that $|\partial F \cap T| = 1$, $[\partial F]$ is not strongly detected in $V_1$, and $\mathcal{D}(V_1) = \{[\partial F]\}$.  
\end{corollary}

\begin{proof}
Theorem \ref{theorem: gluing 2} implies that $\mathcal{D}_{str}(V) = \mathcal{D}_{str}(M_1; \partial V; D_*)$ if we exclude cases (a) and (b), so we need only prove that $\mathcal{D}(V) \setminus \mathcal{D}_{str}(V)$ is finite in the general case.~(Rationality comes for free since irrational slopes are always strongly detected.) Corollary \ref{cor: detection} implies that each $D_j$ is a closed, connected subset of the circle $\mathcal{S}(\partial V_j)$ whose frontier consists of rational slopes. 

If we are not in case (a) or (b) of the corollary's statement, then $\mathcal{D}_{str}(V) = \mathcal{D}_{str}(M_1; \partial V; D_*) \subseteq \mathcal{D}(M_1; \partial V; D_*) = \mathcal{D}(V)$ so that  $\mathcal{D}(V) \setminus \mathcal{D}_{str}(V) =  \mathcal{D}(M_1; \partial V; D_*) \setminus \mathcal{D}_{str}(M_1; \partial V; D_*)$. The latter is finite by Theorem \ref{detected slopes seifert}. In general, Theorem \ref{theorem: gluing 2} implies that $\mathcal{D}(V) \setminus \mathcal{D}_{str}(V)$ is contained in the union of $\{[\lambda_V]\}$, the finite set $\mathcal{D}(M_1; \partial V; D_*) \setminus \mathcal{D}_{str}(M_1; \partial V; D_*)$, and the set slopes $[\alpha]$ which satisfy (a). In the latter case, $V = M_1 \cup_{T_1} V_1$ where $M_1$ is a cable space, $[\alpha]$ is a rational slope such that $M_1([\alpha]) \cong S^1 \times D^2$, and the meridional slope $[\alpha_1]$ of $M_1([\alpha])$ is not strongly detected in $V_1$. Since $M_1([\alpha]) \cong S^1 \times D^2$, $[\alpha_1]$ is the unique slope in $\mathcal{D}(M_1;T_1; \partial V; [\alpha]) = \mathcal{D}_{str}(M_1([\alpha])) = \mathcal{D}(M_1([\alpha]))$. Thus $([\alpha], [\alpha_1])$ is the unique pair of slopes on $\partial M_1$ which is detected in $M_1$ in such a way that $[\alpha]$ is strongly detected. By hypothesis, $[\alpha_1] \in \mathcal{D}(V_1) \setminus \mathcal{D}_{str}(V_1)$, and since $\mathcal{D}(V_1) \setminus \mathcal{D}_{str}(V_1)$ is finite (by our induction hypothesis), the set of all slopes $[\alpha]$ on $\partial V$ which satisfy (a) is finite. Hence $\mathcal{D}(V) \setminus \mathcal{D}_{str}(V)$ is finite.
\end{proof}

\subsection{The space of detected slopes of a graph manifold rational homology solid torus} \label{subsec: slope detection graph} 

As above, $V$ is a graph manifold rational homology solid torus. We denote by $[h] \in \mathcal{S}(\partial V)$ the slope of the Seifert fibre of $M_1$. Note that $\overline{V \setminus M_1}$ is a disjoint union of $r \geq 0$ graph manifold rational homology solid tori $V_1, \ldots, V_{r}$. 

The next theorem and its corollary yield a refined version of Theorem \ref{detection in graph manifolds qhst}.

\begin{theorem} \label{refined detection in graph manifolds qhst}
Let $V$ be as above. Assume that if $[\alpha]$ is rational and horizontal and $M_1([\alpha])$ fibres over the circle with fibre $F$ such that $|\partial F| = 1$ and $[\partial F]$ is not strongly detected in $V_1 = \overline{V \setminus M_1}$, then neither of the following occurs 
\vspace{-.3cm} 
\begin{itemize}
\item $M_1$ is a cable space and $\Delta([\alpha], h) = 1$, so $M_1([\alpha]) \cong S^1 \times D^2$;

\item $M_1$ has base orbifold an annulus with cone points and $\mathcal{D}(V_1) = \{[\partial F]\}$, so $[\alpha] = [\lambda_V]$. 

\end{itemize}
\vspace{-.3cm}
Then

$(1)$ $\mathcal{D}(V)$ is a closed, connected subset of the circle $\mathcal{S}(\partial V)$ whose frontier consists of rational slopes. Further, $[\lambda_V] \in \mathcal{D}(V)$. 

$(2)$ If $M$ has base orbifold $Q(a_1, \ldots, a_n)$ then 

\indent \hspace{.5cm} $(a)$ $\mathcal{D}(V) = 
\left\{ \begin{array}{ll} \{[h]\} & \hbox{ if there is no $j$ such that } [h] \in \mathcal{D}(V_j)  \\ & \\ \mathcal{S}(\partial V) & \hbox{ if there is a $j$ such that } [h]  \in \mathcal{D}(V_j) \end{array} \right.$

\indent \hspace{.5cm} $(b)$ $\mathcal{D}_{str}(V) = 
\left\{ \begin{array}{ll}\{[h]\} & \hbox{ if }  M \cong N_2\\ & \\ \emptyset &  \hbox{ if there is no $j$ such that } [h] \in \mathcal{D}(V_j)  \hbox{ and } M \not \cong N_2  \\ & \\ \mathcal{S}(\partial V) \setminus \{[h]\} & \hbox{ if there is a $j$ such that } [h]  \in \mathcal{D}(V_j)  \end{array} \right.$

$(3)$ If $M$ has base orbifold $P(a_1, \ldots, a_n)$ and  

\indent \hspace{.5cm} $(a)$ $[h] \not \in \mathcal{D}(V_j)$ for each $j$, then $\mathcal{S}(T_r; S_*)$ is a non-empty interval of horizontal slopes. \\ \indent \hspace{1.1cm} 
Further either 

\indent \hspace{1.1cm} $(i)$ $\mathcal{D}_{str}(V) = \mathcal{D}(V)$ consists of a single rational slope, or 

\indent \hspace{1.1cm} $(ii)$ $(\mathcal{D}(V), \mathcal{D}_{str}(V) ) \cong ([0,1], (0,1))$.

\indent \hspace{.5cm} $(b)$ $[h] \in \mathcal{D}(V_j)$ for exactly one $j$ then $\mathcal{D}(V)$ is a closed, connected subset of $\mathcal{S}(\partial V)$ containing  \\ \indent \hspace{1.1cm} $\{[h]\}$. Further, $\mathcal{D}_{str}(V) = \hbox{int}(\mathcal{D}(V)) \setminus \{[h]\}$

\indent \hspace{.5cm} $(c)$ $[h] \in \mathcal{D}(V_j)$ for exactly two or more $j$ then $\mathcal{D}(V)  = \mathcal{S}(\partial V)$ and $\mathcal{D}_{str}(V) = \mathcal{S}(\partial V) \setminus \{[h]\}$.

\end{theorem}

\begin{proof} 
If $M_1 = V$ the result follows from \cite{BC}. Suppose that the theorem holds for all graph manifold rational homology solid tori having fewer Seifert fibred pieces than $V$. Set $D_j = \mathcal{D}(V_j)$ and $D_* = D_1 \times \ldots \times D_{r-1}$. Corollaries \ref{cor: detection} and \ref{cor: strong detection} show that $\mathcal{D}(V) = \mathcal{D}(M_1; \partial V; D_*)$ and $\mathcal{D}_{str}(V) = \mathcal{D}_{str}(M_1; \partial V; D_*)$. Since the frontier of each $D_j$ is rational and $[\lambda_V] \in \mathcal{D}(V)$ (Corollary \ref{cor: detection}), Theorem \ref{refined detection in graph manifolds qhst} now follows from the Theorem \ref{detected slopes seifert}.

%Let $F$ be a compact, connected, orientable, incompressible surface in $V$ whose oriented boundary consists of like-oriented curves of slope $[\lambda_V]$ chosen to intersect the JSJ tori of $V$ minimally amongst all such surfaces. Then for each $j$, $F \cap \partial V_j$ is either empty or consists of curves of slope $[\lambda_{V_j}]$ while $F \cap M_1$ is an essential surface in $M_1$, and so is either horizontal or vertical. In the former case, $M_1$ fibres over the circle with fibre $F \cap M_1$. Thus $([\lambda_{V_1}], \ldots, [\lambda_{V_{r-1}}], [\lambda_V])$ is foliation detected in $M_1$. It follows from Corollary \ref{cor: detection} and our induction hypothesis that $[\lambda_V]$ is foliation detected in $V$. On the other hand, if $F \cap M_1$ is vertical, \cite[Proposition 6.5]{BC} implies that $[\lambda_V]$ is foliation detected in $V$.   
\end{proof}

\begin{corollary} \label{difference} Let $V$ be as above.  

$(1)$ $\mathcal{D}(V) \setminus \mathcal{D}_{str}(V)$ is finite and if $\mathcal{D}(V)$ is a non-degenerate interval, then its endpoints are rational and not contained in $\mathcal{D}_{str}(V)$. 

$(2)$ Assume that if $[\alpha]$ is rational and horizontal, $M_1([\alpha])$ fibres over the circle with fibre $F$ such that $|\partial F| = 1$ and $[\partial F]$ is not strongly detected in $V_1 = \overline{V \setminus M_1}$, then neither of the following occurs 
\vspace{-.3cm}
\begin{itemize}
\item $M_1$ is a cable space and $\Delta([\alpha], h) = 1$ $($so $M_1([\alpha]) \cong S^1 \times D^2$$)$;  

\vspace{.2cm} \item $M_1$ has base orbifold an annulus with cone points and $\mathcal{D}(V_1) = \{[\partial F]\}$, $($so $[\alpha] = [\lambda_V]$$)$; 
\end{itemize} 
\vspace{-.3cm}

$(a)$ If $M$ has base orbifold $Q(a_1, \ldots, a_n)$ then 
$$\mathcal{D}(V) \setminus \mathcal{D}_{str}(V) = \left\{\begin{array}{ll} \emptyset & \hbox{ if } V \cong N_2\\ & \\ \{[h]\} & \hbox{ if } V \not \cong N_2 \end{array} \right.$$

$(b)$ If $M$ has base orbifold $P(a_1, \ldots, a_n)$ then 
$$\mathcal{D}(V) \setminus \mathcal{D}_{str}(V) = \left\{\begin{array}{ll} \emptyset & \hbox{ if } \mathcal{D}(V) = \{[\lambda_V]\}  \\ & \\ \hbox{fr}(\mathcal{D}(V)) & \hbox{ if } \mathcal{D}(V) \ne \{[\lambda_V]\} \hbox{ and } [h] \not \in \mathcal{D}(V_j) \hbox{ for each } j \\ & \\ \hbox{fr}(\mathcal{D}(V)) \cup \{[h]\} & \hbox{ if } \mathcal{D}(V) \ne \{[\lambda_V]\} \hbox{ and } [h] \in \mathcal{D}(V_j) \hbox{ for some } j \end{array} \right.$$

$(c)$ If $[h] \in \mathcal{D}(V) \setminus \mathcal{D}_{str}(V)$ then either $M$ has base orbifold $Q(a_1, \ldots, a_n)$ and $V \not \cong N_2$ or $M$ has base orbifold $P(a_1, \ldots, a_n)$ and $[h] \in \mathcal{D}(V_j)$ for some $j$.  
\qed 
\end{corollary}

\section{Graph manifold rational homology solid tori with degenerate sets of detected slopes} \label{sec: degenerate}

In this section we consider the degenerate case of graph manifold rational homology solid tori $V$ for which $\mathcal{D}(V)$ is a point. In this case, $\mathcal{D}(V) = \{[\lambda_V]\}$ by Corollary \ref{cor: detection}. Throughout we take $M$ to be the piece of $V$ which contains $\partial V$. 

Suppose that $\partial V, T_1, \ldots, T_r$ are the boundary tori of $M$ and $V_1, \ldots, V_r$ the graph manifold rational homology solid tori in $V$ where $T_i = \partial V_i$. Set
$$D_* = \mathcal{D}(V_1) \times \ldots \times \mathcal{D}(V_r)$$ 
and recall that $v( D_*) = |\{ j : [h] \in \mathcal{D}(V_j) \}|$.

\begin{lemma} \label{a point outer piece} 
Let $V$ be a graph manifold rational homology solid torus which admits a co-oriented taut foliation and suppose that $M, T_i, V_i$ are as above.  

$(1)$ If $M$ has base orbifold $Q(a_1, \ldots , a_n)$, then $\mathcal{D}(V) = \{[\lambda_V]\}$ if and only if $v(D_*) = 0$. 

$(2)$  If $M$ has base orbifold $P(a_1, \ldots , a_n)$ and $[\lambda_V]$ is horizontal in $M$, then $\mathcal{D}(V) = \{[\lambda_V]\}$ if and only if $v(D_*) = 0$, $\mathcal{D}(V_i) = \{[\lambda_{V_i}]\}$ for each $i$, and $\mathcal{D}(M; T; ([\lambda_{V_1}], \ldots, [\lambda_{V_r}])) = \{[\lambda_V]\}$. In this case $M$ fibres over the circle with fibre which detects $([\lambda_V], [\lambda_{V_1}], \ldots, [\lambda_{V_r}])$. 

$(3)$ If $M$ has base orbifold $P(a_1, \ldots , a_n)$ and $[\lambda_V]$ is vertical in $M$, then $\mathcal{D}(V) = \{[\lambda_V]\}$ if and only if $v(D_*) = 1$ and if $i$ is the index for which $[h] \in \mathcal{D}(V_i)$, then $\mathcal{D}(V_i) = \{[h]\}$. 

\end{lemma}

\begin{proof}
Recall from Corollary \ref{cor: detection} that $\mathcal{D}(V) = \mathcal{D}(M; T; D_*)$.  

When $M$ has base orbifold $Q(a_1, \ldots , a_n)$, Theorem \ref{detected slopes seifert}  implies that 
$$\mathcal{D}(M; T; D_*) =
\left\{ \begin{array}{ll} \{[h]\} & \hbox{ if } v(D_*) = 0 \\ & \\ \mathcal{S}(T_r) & \hbox{ if } v(D_*) > 0 \end{array} \right.$$ 
which shows that assertion (1) holds. 

Suppose that $M$ has base orbifold $P(a_1, \ldots , a_n)$ and $[\lambda_V]$ is horizontal in $M$. If $\mathcal{D}(V) = \{[\lambda_V]\}$, then Theorem \ref{theorem: gluing} combines with the method of proof of Case 2 of the proof of \cite[Lemma 10.3]{BC} to show that if the closed, connected subset $\mathcal{D}(V_i)$ of the circle $\mathcal{S}(V_i)$ is not $\{[\lambda_{V_i}]\}$, then $\mathcal{D}(V) \ne \{[\lambda_V]\}$. Thus $\mathcal{D}(V_i) = \{[\lambda_{V_i}]\}$ for each $i$. Hence $\{[\lambda_V]\} = \mathcal{D}(V) = \mathcal{D}(M; T; ([\lambda_{V_1}], \ldots, [\lambda_{V_r}])) = \{[\lambda_V]\}$. Further, Theorem \ref{detected slopes seifert} shows that $v(D_*) \leq 1$. The case that $v(D_*) = 1$ is impossible as it implies that $[h] \in \mathcal{D}(M; T; D_*) = \{[\lambda_V]\}$ (cf. Lemma 6.5 of \cite{BC}). Thus $v(D_*) = 0$. The argument of the proof of \cite[Lemma 10.4]{BC} then shows that $M$ fibres over the circle with fibre which detects $([\lambda_V], [\lambda_{V_1}], \ldots, [\lambda_{V_r}])$. Conversely, if $v(D_*) = 0$, $\mathcal{D}(V_i) = \{[\lambda_{V_i}]\}$ for each $i$, and $\mathcal{D}(M; T; ([\lambda_{V_1}], \ldots, [\lambda_{V_r}])) = \{[\lambda_V]\}$. Then $\mathcal{D}(V) = \mathcal{D}(M; T; D_*) = \mathcal{D}(M; T; ([\lambda_{V_1}], \ldots, [\lambda_{V_r}])) = \{[\lambda_V]\}$. Thus assertion (2) holds. 

Finally suppose that $M$ has base orbifold $P(a_1, \ldots , a_n)$ and $[\lambda_V]$ is vertical in $M$. In this case, Lemma 6.5 of \cite{BC} implies that $v(D_*) \geq 1$. If $\mathcal{D}(V) = \{[\lambda_V]\}$, then Theorem \ref{theorem: gluing}(3) shows that $v(D_*) = 1$ and combines with the method of proof of Case 1 of the proof of \cite[Lemma 10.3]{BC} to show that if  and if $i$ is the index for which $[h] \in \mathcal{D}(V_i)$, then $\mathcal{D}(V_i) = \{[h]\}$. Conversely, suppose that $v(D_*) = 1$ and if  and only if $i$ is the index for which $[h] \in \mathcal{D}(V_i)$, then $\mathcal{D}(V_i) = \{[h]\}$. Lemma 6.5 of \cite{BC} implies that $\{[\lambda_V]\} \subseteq \mathcal{D}(V) = \mathcal{D}(M; T; D_*) = \{[h]\}$. Thus $\{[\lambda_V]\} = \mathcal{D}(V) = \{[h]\}$. Thus assertion (3) holds.  
\end{proof}

When $\mathcal{D}(V) = \{[\lambda_V]\}$ and $V$ admits a horizontal co-oriented foliation, the topology if $V$ is strongly restricted.

\begin{proposition} \label{a point horizontal} 
Let $V$ be a graph manifold rational homology solid torus which admits a horizontal co-oriented foliation and suppose that $\mathcal{D}(V) = \{[\lambda_V]\}$. 

$(1)$ If $U$ is a union of pieces of $V$ which is connected and has connected boundary, then $\mathcal{D}(U) = \{[\lambda_U]\}$. Moreover, $[\lambda_U]$ is horizontal in the piece of $U$ incident to $\partial U$. 

$(2)$ $V$ fibres horizontally over the circle with fibres which intersect the JSJ tori of $V$ in foliations by circles. 

\end{proposition}

\begin{proof}
We proceed by induction on the number $p$ of pieces of $V$. 

When $p = 1$, $V$ is Seifert and the result follows from \cite[Lemma 2.2 and the method of proof of Proposition 6.13]{BC}. 

Suppose that $p> 1$ and recall that $M$ is the outer piece of $V$ with $\partial M = \partial V \cup T_1 \cup \ldots \cup T_r$ and $V = M \cup V_1 \cup \ldots \cup V_r$ where $V_1, \ldots , V_r$ are disjoint rational homology solid tori with $\partial V_i = T_i$. 

The method of proof of Lemma 10.3 of \cite{BC} shows that if some $\mathcal{D}(V_i) \ne \{[\lambda_{V_i}]\}$ then  $\mathcal{D}(V) \ne \{[\lambda_{V}]\}$. Thus $\mathcal{D}(V_i) = \{[\lambda_{V_i}]\}$ for each $i$ and so by induction assertion (1) holds for $V$. Induction also shows that each $V_i$ fibres horizontally over the circle with fibres which intersect $\partial V_i$ and the JSJ tori of $V_i$ in foliations by circles. Clearly, the fibration on $V_i$ detects $[\lambda_{V_i}]$. The restriction of $\mathcal{F}$ to $M$ detects a horizontal element $[\alpha_*] \in \mathcal{S}(M)$ and from above we can suppose that  
$$[\alpha_*] = ([\lambda_V], [\lambda_{V_1}], \ldots, [\lambda_{V_r}])$$
On the other hand, there is a unique horizontal, primitive class $\beta \in H_1(\partial V)$ which is a rational combination of $\lambda_{V_1}, \ldots, \lambda_{V_r}$. Hence there is a horizontal fibration on $M$ which detects the slopes $[\beta]$ on $T$ and $[\lambda_{V_i}]$ on $T_i$. We can glue the fibration on $M$ and those given inductively on the $V_i$ together to produce a fibration on $V$ which detects $[\beta]$. Our hypothesis that $\mathcal{D}(V) = \{[\lambda_V]\}$ then implies that $\beta = \pm \lambda_V$, which shows that assertion (2) holds. 
\end{proof}

Though Proposition \ref{a point horizontal} fails when $V$ admits no horizontal co-oriented foliation, the hypothesis that $\mathcal{D}(V) = \{[\lambda_V]\}$ still places strong conclusions on the topology $V$ in certain circumstances. 

\begin{proposition} \label{a point general} 
Let $V$ be a graph manifold rational homology solid torus and suppose that $\mathcal{D}(U) = \{[\lambda_U]\}$ for each graph manifold rational homology solid torus $U \subseteq V$ which is a union of pieces of $V$. Then there are smooth admissible foliations on the pieces of $V$ which determine a gluing coherent family of slopes for $V$ and which restrict to hyperbolic foliations on each boundary component of each piece of $V$.
\end{proposition}

\begin{proof}
Recall that $M$ is the outer piece of $V$ with $\partial M = \partial V \cup T_1 \cup \ldots \cup T_r$ and $V = M \cup V_1 \cup \ldots \cup V_r$ where $V_1, \ldots , V_r$ are disjoint rational homology solid tori with $\partial V_i = T_i$. By hypothesis, $\mathcal{D}(V_i) = \{[\lambda_{V_i}]\}$, so by induction, the proposition holds for each $V_i$. Thus to finish the proof it suffices to construct a smooth foliation on $M$ which detects $[\alpha_*] = ([\lambda_V], [\lambda_{V_1}], \ldots, [\lambda_{V_r}])$ and which restricts to a hyperbolic foliation on each component of $\partial M$. 

When $V$ is Seifert (i.e. $r = 0$) we consider two cases. 
\vspace{-.2cm} 
\begin{itemize}
\item {\bf $[\lambda_V]$ is horizontal}: In this case $V$ has base orbifold $D^2(a, a)$ for some $a \geq 2$ by \cite[Proposition 6.5 and Proposition A.4(4)]{BC}. Arguing as in \S 2.2.3 of \cite{BC} shows that $V$ fibres over the circle with fibre a surface with $a$ boundary components. Now apply Proposition \ref{compact} to complete the proof. 

\vspace{.2cm} \item {\bf $[\lambda_V]$ is vertical}: In this case $V$ has base orbifold $Q_0(a_1, \ldots, a_n)$ for some $n \geq 0$ by \cite[Proposition 6.5]{BC}. Any admissible foliation on $V$ which detects $[\lambda_V]$ contains a leaf which is a vertical annulus by \cite[Proposition 6.2]{BC}, so we can apply Proposition \ref{compact} to complete the proof. 
 
\end{itemize}
\vspace{-.2cm} Thus the lemma holds when $r = 0$.  

Assume that $r \geq 1$. Since $\mathcal{D}(V) = \{[\lambda_{V}]\}$ and $\mathcal{D}(V_i) = \{[\lambda_{V_i}]\}$ for each $i$, there is an admissible foliation $\mathcal{F}$ on $M$ which detects $[\alpha_*]$. If $v([\alpha_*]) > 0$, there is a leaf of $\mathcal{F}$ which is a vertical annulus (\cite[Proposition 6.2]{BC}). If $v([\alpha_*]) = 0$, Lemma \ref{a point outer piece} shows that $M$ fibres over the circle with fibre which detects $[\alpha_*]$. In either case we can apply Proposition \ref{compact} to complete the proof of the proposition. 
\end{proof}

\begin{proposition} \label{point interval hyperbolic}
Let $V$ be a graph manifold rational homology solid torus and suppose that $\mathcal{D}(U) = \{[\lambda_U]\}$ for each graph manifold rational homology solid torus $U \subseteq V$ which is a union of pieces of $V$. Then for each odd $k \geq 1$ there is a smooth foliation on $V$ made up of admissible foliations on the pieces of $V$ and which is $k$-interval hyperbolic on each boundary component of each piece of $V$. 
\end{proposition}

\begin{proof}
The proof is similar to that of the previous proposition. Simply replace the use of Proposition \ref{compact} there by that of Proposition \ref{compact 2} here. 
\end{proof}

\begin{remark}
{\rm A similar result holds with $k$-interval hyperbolic replaced by $k$-hyperbolic.}
\end{remark}

\section{Block decompositions of foliated graph manifold rational homology $3$-spheres} \label{sec: block decompositions}

In this section we take $W$ to be a graph manifold rational homology $3$-sphere which admits a co-oriented taut foliation. 

\begin{definition}
{\rm An {\it admissible foliation} on a Seifert fibred manifold with non-empty boundary which is a piece a graph manifold rational homology $3$-sphere is either a smooth foliation with a compact leaf, or a smooth foliation of the form $\mathcal{F}(\rho)$ for some homomorphism $\rho : \pi_1(M) \rightarrow \widetilde{PSL}(2, \mathbb{R})_j$, or a smooth foliation constructed from one of these as in \S \ref{sec: operations}. }
\end{definition}

A {\it block decomposition $\mathcal{B}$} of $W$  is a splitting of $W$ along a subset of its JSJ tori  into connected submanifolds ({\it blocks}) such that
\vspace{-.2cm} 
\begin{itemize}

\item $W$ is the union of the blocks; 

\vspace{.2cm} \item the pieces of $W$ are endowed with admissible foliations which determine a gluing coherent family of rational slopes on the JSJ tori of $W$; 

\vspace{.2cm} \item the admissible foliations piece together to form a smooth foliation on each block which is either elliptic or hyperbolic on the block's boundary components and either elliptic or $k$-interval hyperbolic for some {\it odd} $k$ on the block's JSJ tori. 

\end{itemize}
It follows from \cite[Theorem 9.5 and Proposition A.8]{BC} that a graph manifold rational homology $3$-sphere admits a block decomposition if and only if it admits a co-oriented taut foliation. 

Given a block decomposition $\mathcal{B}$ of $W$ and a JSJ torus $T$ of $W$, we say that $T$ has {\it type e-e} (for elliptic-elliptic), {\it ih-ih, h-h}, or {\it e-h} depending on how the foliations on the two pieces of $W$ incident to $T$ restrict to $T$. We call a JSJ torus $T$ of $W$ {\it good} (with respect to $\mathcal{B}$) if it is a JSJ torus of a block of $\mathcal{B}$ or if it is on the boundary of a block and has type e-e or h-h. Otherwise we call $T$ {\it bad}.  Our goal in this section is to prove the following proposition. 

\begin{proposition} \label{no bad}
Let $W$ be a graph manifold rational homology $3$-sphere which admits a co-oriented taut foliation. Then there is a block decomposition of $W$ with no bad tori. 
\end{proposition}

\begin{remark}
{\rm We can piece together a block decomposition of $W$ without bad tori along any JSJ torus of type e-e to produce a new block decomposition whose foliations restrict to hyperbolic foliations of matching slopes on the boundary tori of the blocks. On the other hand, there are graph manifold rational homology $3$-spheres which admit co-oriented taut foliations but admit no family of  transversally projective co-oriented taut foliations on its pieces which restrict to hyperbolic foliations of matching slopes on its JSJ tori. }
\end{remark}

To each block decomposition $\mathcal{B}$ of $W$ we assign a complexity 
$$c(W, \mathcal{B}) = (\# \hbox{ pieces of } W,   \# \hbox{ bad tori of $\mathcal{B}$}, \# \hbox{ e-e tori of $\mathcal{B}$}) \in \mathbb N \oplus \mathbb W  \oplus \mathbb W$$ 
where $\mathbb N$ denotes the natural numbers $\{1, 2, 3, \ldots\}$ and $\mathbb W$ denotes the whole numbers $\{0\} \cup \mathbb N$.

Order $\mathbb N \oplus \mathbb W  \oplus \mathbb W$ lexicographically from right to left (so that the rightmost factor in the sum is smallest).  We prove the Proposition \ref{no bad} by contradiction, so assume that it is false. Suppose that $W$ and its block decomposition $\mathcal{B}$ are chosen among all block decompositions of counterexamples to Proposition \ref{no bad} to minimize $c(W, \mathcal{B})$. No boundary torus $T$ of a block of $\mathcal{B}$ has type e-e as otherwise we could produce a new block decomposition $\mathcal{B}'$ of $W$ with $c(W, \mathcal{B}') < c(W, \mathcal{B})$ by gluing together the two blocks incident to $T$ and their foliations. Thus  {\it the boundary tori of the blocks of $\mathcal{B}$ are either good and type h-h or bad and type e-h.}

\begin{lemma} \label{conditions on JSJ tori}
Let $T$ be a JSJ torus of $W$ and let $U$ and $V$ be the rational homology solid tori obtained by splitting $W$ open along $T$.  

$(1)$ The interior of $\mathcal{D}(U) \cap \mathcal{D}(V)$ is empty.  

$(2)$ If neither $\mathcal{D}(U)$ nor $\mathcal{D}(V)$ is a point, then $T$ is of type h-h.     

\end{lemma}

\begin{proof}
Corollary \ref{cor: strong detection} implies that if $\hbox{int}( \mathcal{D}(U) \cap \mathcal{D}(V) ) \ne \emptyset$ we may choose a rational horizontal slope $[\alpha] \in \mathcal{S}(T)$ which
\begin{itemize} 
\vspace{-.2cm} \item is contained in $\mathcal{D}_{str}(U) \cap \mathcal{D}_{str}(V)$; 
\vspace{.2cm} \item has distance at least $2$ to the Seifert fibres of the outer pieces of $U$ and $V$; 
\vspace{.2cm} \item is not the rational longitude of either $U$ or $V$. 
\end{itemize}
\vspace{-.2cm} 
Then the result of Dehn filling $U$ and $V$ along $[\alpha]$ are graph manifold rational homology $3$-spheres admitting co-orientable taut foliations (\cite[Theorem 9.5]{BC}). Further, the pieces of the Dehn fillings $U([\alpha])$ and $V([\alpha])$ correspond in the obvious way to those of $W$. Since $U$ and $V$ have fewer pieces than $W$, the minimality of $c(W, \mathcal{B})$ implies that each admits a block decomposition with no bad tori. Further, the foliations can be chosen to be transverse to the cores of the filling solid tori in $U([\alpha])$ and $V([\alpha])$ (\cite[Lemma 6.12]{BC}). But then we can produce a block decomposition of $W$ with no bad tori, contrary to our assumptions. Thus assertion (1) holds.

Suppose that neither $\mathcal{D}(U)$ nor $\mathcal{D}(V)$ is a point. Then Corollary \ref{difference} implies that $\mathcal{D}(U)$ and $\mathcal{D}(V)$ are non-degenerate intervals and $\mathcal{D}(U) \cap \mathcal{D}(V) = \partial \mathcal{D}(U) \cap \partial \mathcal{D}(V) \subseteq (\mathcal{D}(U) \setminus \mathcal{D}_{str}(U)) \cap (\mathcal{D}(V) \setminus \mathcal{D}_{str}(V)) $. In either case, the foliations from the pieces of $U$ and $V$ incident to $T$ are hyperbolic on $T$, so $T$ has type h-h. 
\end{proof}

\begin{lemma} \label{point to point}
Let $T$ be a JSJ torus of $W$ and let $V$ be one of the rational homology solid tori in $W$ with boundary $T$. Suppose that $\mathcal{D}(V) = \{[\lambda_V]\}$. 
If $V_1$ is a union of pieces of $V$ which is connected and has connected boundary, then $\mathcal{D}(V_1) = \{[\lambda_{V_1}]\}$. 
\end{lemma}

\begin{proof}
We prove the lemma by induction on the number $p$ of pieces of $V$. When $p = 1$, $V$ is Seifert and the lemma holds by hypothesis. Suppose that $p > 1$ and let $M$ be the outer piece of $V$. Write $\partial M = \partial V \cup T_1 \cup \ldots \cup T_r$ and $V = M \cup V_1 \cup \ldots \cup V_r$ where $V_1, \ldots , V_r$ are disjoint rational homology solid tori with $\partial V_i = T_i$. By induction, it suffices to show that $\mathcal{D}(V_i) = \{[\lambda_{V_i}]\}$ for each $i$. Set $D_* = \mathcal{D}(V_1) \times \ldots \times \mathcal{D}(V_r)$. 

If $M$ has base orbifold $Q(a_1, \ldots , a_n)$, then $[\lambda_V] = [h]$. Set $U_1 = \overline{W \setminus V_1}$, so that $W = U_1 \cup_{T_1} V_1$, and suppose that $\mathcal{D}(V_1) \ne \{[\lambda_{V_1}]\}$. Then the interior of $\mathcal{D}(V_1)$ is non-empty by Theorem \ref{detection in graph manifolds qhst}. Proposition 6.5(1) of \cite{BC} implies that for each element $([\alpha_1], \ldots, [\alpha_r])$ of $\mathcal{D}_*$, $([h], [\alpha_1], \ldots, [\alpha_r])$ is foliation detected in $M$, and since $[\lambda_{V_j}] \in \mathcal{D}(V_j)$ for each $j$, for each $[\alpha_1] \in \mathcal{D}(V_1)$ there is a gluing coherent family of slopes on $U_i$ which extends $([h], [\alpha_1], [\lambda_{V_2}], \ldots, [\lambda_{V_r}])$. Theorem \ref{theorem: gluing} then shows that $\mathcal{D}(V_1) \cap \mathcal{D}(U_1) \supseteq \mathcal{D}(V_1)$ and hence has a non-empty interior, contrary to Lemma \ref{conditions on JSJ tori}(1). Thus $\mathcal{D}(V_1) = \{[\lambda_{V_1}]\}$ and similarly $\mathcal{D}(V_i) = \{[\lambda_{V_i}]\}$ for each $i$. 

If $M$ has base orbifold $P(a_1, \ldots , a_n)$ and $[\lambda_V]$ is vertical in $M$, then Lemma \ref{a point outer piece}(3) implies that $v(D_*) = 1$ and if $[h] \in \mathcal{D}(V_i)$, then $\mathcal{D}(V_i) = \{[h]\}$. Without loss of generality we suppose that $[h] \in \mathcal{D}(V_1)$. As in the previous case, Lemma \ref{conditions on JSJ tori}(1) combines with \cite[Proposition 6.5]{BC} to show that $\mathcal{D}(V_i) = \{[\lambda_{V_i}]\}$ for $2 \leq i \leq r$, which completes the proof in this case. Finally if $M$ has base orbifold $P(a_1, \ldots , a_n)$ and $[\lambda_V]$ is horizontal in $M$, then Lemma \ref{a point outer piece}(2) implies that $\mathcal{D}(V_i) = \{[\lambda_{V_i}]\}$ for each $i$. 
\end{proof}

Lemma \ref{point to point} combines with Propositions \ref{a point general} and \ref{point interval hyperbolic} to yield the following lemma. 

\begin{lemma} \label{point to hyp or int-hyp}
Let $T$ be a JSJ torus of $W$ and let $V$ be one of the rational homology solid tori in $W$ with boundary $T$. If $\mathcal{D}(V) = \{[\lambda_V]\}$ then 

$(1)$ there are smooth foliations on the pieces of $V$ which determine a gluing coherent family of slopes for $V$ and which restrict to hyperbolic foliations on each boundary component of each piece of $V$.

$(2)$ for each odd $k \geq 1$ there is a smooth foliation on $V$ which is $k$-interval hyperbolic on each boundary component of each piece of $V$. 
\qed
\end{lemma}

\begin{proof}[Proof of Proposition \ref{no bad}] By hypothesis, there is at least one bad torus $T$ on the boundary of some block of $\mathcal{B}$. Let $U$ and $V$ be the rational homology solid tori obtained by splitting $W$ open along $T$. We suppose that $T$ is elliptic to the $V$-side and hyperbolic to the $U$-side. Note that $\mathcal{D}(V) \ne \{[\lambda_V]\}$ since otherwise we could use Lemma \ref{point to hyp or int-hyp}(1) to reduce the complexity of $(W, \mathcal{B})$. Thus $\mathcal{D}(U) = \{[\lambda_U]\}$ by Lemma \ref{conditions on JSJ tori}(2). 

Let $M$ be the outer piece of $V$ and suppose that $T, T_1, \ldots, T_r$ are the boundary tori of $M$. Let $V_1, \ldots, V_r$ be the graph manifold rational homology solid tori in $V$ such that $T_i = \partial V_i$.  

If $T_i$ is hyperbolic to the $M$-side for some $i$ we could use Lemma \ref{point to hyp or int-hyp} to find a smooth co-oriented, taut foliation $\mathcal{F}_U$ on $U$ which is $k$-interval hyperbolic on $T$ for some odd $k$. Then we could use Proposition \ref{elliptic to ih} to find a smooth foliation on $M$ which agrees with $\mathcal{F}_U$ on $T$ and coincides with the original foliation on $M$ near $T_1 \cup \ldots \cup T_r$. But then we could reduce the complexity of our block decomposition, which is impossible. Thus each $T_i$ is either elliptic or interval hyperbolic to the $M$-side. 

Set $U_i = \overline{W \setminus V_i}$ and suppose that $\mathcal{D}(U_i) = \{[\lambda_{U_i}]\}$ for some $i$.  By the previous paragraph $T_i$ has type e-h, e-e or ih-ih, with $T_i$ elliptic to the $M$ side if it is e-h.  In each case we can find a new block decomposition of $W$ with smaller complexity:

\vspace{-.2cm} 
\begin{itemize}

\item If $T_i$ has type e-h we do this by applying Lemma \ref{point to hyp or int-hyp}(1) to $U_i$ to obtain a new block decomposition of $W$ which agrees with $\mathcal{B}$ on $V_i$ and  for which there are no bad tori contained in $U_i$;

\vspace{.2cm} \item If $T_i$ has type ih-ih we proceed as in the previous case, but apply Lemma \ref{point to hyp or int-hyp}(2) to $U_i$ to obtain a new block decomposition of $W$ which agrees with $\mathcal{B}$ on $V_i$ and for which there are no bad tori contained in $U_i$;

\vspace{.2cm} \item If $T_i$ has type e-e we can apply Lemma \ref{point to hyp or int-hyp}(1) to $U_i$ to find a new block decomposition of $W$ which agrees with $\mathcal{B}$ on $V_i$ and for which the only bad torus which is contained in $U_i$ is $T_i$. This new block decomposition has no more bad tori than $\mathcal{B}$, but has fewer e-e tori.  

\end{itemize}
\vspace{-.3cm} 
This contradicts our choice of $(W, \mathcal{B})$, so $\mathcal{D}(U_i) \ne \{[\lambda_{U_i}]\}$ for each $i$. Since no $T_i$ is hyperbolic to the $M$-side, Lemma \ref{conditions on JSJ tori}(2) now implies that $\mathcal{D}(V_i) = \{[\lambda_{V_i}]\}$ for each $i$. 

If there is an admissible foliation on $M$ which detects $([\lambda_U], [\lambda_{V_1}], \ldots, [\lambda_{V_r}])$ but has no compact leaves, then $([\lambda_U], [\lambda_{V_1}], \ldots, [\lambda_{V_r}])$ is horizontal in $M$ and Proposition \ref{elliptic to ih} implies that there is a smooth foliation on $M$ which is either hyperbolic or $1$-interval hyperbolic on each component of $\partial M$. Applying Lemma \ref{point to hyp or int-hyp} appropriately to each of the $V_i$ shows that we can construct a block decomposition of $W$ without bad tori, contrary to our choices. 

If $([\lambda_U], [\lambda_{V_1}], \ldots, [\lambda_{V_r}])$ is detected by an admissible foliation in $M$ with a compact leaf $F$ such that $|\partial M| + |\partial F| \geq 3$, then Proposition \ref{compact}(1) and Lemma \ref{point to hyp or int-hyp}(1) imply that there is a block decomposition of $W$ without bad tori, again a contradiction. 

Finally if $([\lambda_U], [\lambda_{V_1}], \ldots, [\lambda_{V_r}])$ is detected by an admissible foliation in $M$ with a compact leaf $|F|$ such that $|\partial M| + |\partial F| = 2$, then $|\partial M| = |\partial F| = 1$. But this is impossible since it would imply that $V = M$ and $[\lambda_U]$ has the same slope as $[\partial F]$, which is $[\lambda_V]$, contrary to the fact that $W$ is a rational homology $3$-sphere. This final contradiction completes the proof of Proposition \ref{no bad}. 
\end{proof}

A positive answer to the following question would imply that a graph manifold rational homology $3$-sphere which admits a co-oriented taut foliation admits a smooth co-oriented taut foliation. 

\begin{question} \label{q: matching hyperbolic behaviour}
Let $W$ be a graph manifold rational homology $3$-sphere which admits a co-oriented taut foliation. Is there is a block decomposition of $W$ with no bad tori such that the number of closed leaves on a boundary component $T$ of a block is the same for both foliations on the two blocks incident to $T$? 
\end{question}

\section{Graph manifolds with left-orderable fundamental groups are not L-spaces} \label{sec: lo --> nls}

In this section we prove Theorem \ref{LO implies NLS}: {\it If $W$ is a graph manifold rational homology $3$-sphere with a left-orderable fundamental group, then it is not an L-space}. We do this by constructing a closed $2$-form $\omega$ and contact structures $\xi^+, \xi^-$ of opposite sign on $W$ such that $\omega | \xi^{\pm} > 0$. To see why this is sufficient, fix a smooth $1$-form $\alpha$ on $W$ which evaluates positively on a vector field tranverse to both $\xi^+$ and $\xi^-$ and let $p$ and $q$ be the projections of $V = W \times [-1, 1]$ to $W$ and $[-1, 1]$. For $\epsilon> 0$ consider the closed $2$-form 
$$\omega_{V} = p^*(\omega) + \epsilon d(q p^*(\alpha))$$
on $V$. For $\epsilon$  small, $\omega_V \wedge \omega_V > 0$ and so $\omega_V$ is a symplectic form on $V$ which is positive on $\xi^+ \times \{1\}$ and $\xi^- \times \{-1\}$. Thus $(W, \xi^+)$ is semi-fillable and therefore $W$ is not an L-space \cite[Theorem 1.4]{OSz2004-genus}. 

\begin{proof}[Proof of Theorem \ref{LO implies NLS}] 
From above we are reduced to constructing a closed $2$-form $\omega$ and contact structures $\xi^+, \xi^-$ of opposite sign on $W$ such that $\omega | \xi^{\pm} > 0$.

Since $W$ has a left-orderable fundamental group, it admits a co-oriented taut foliation (\cite[Theorem 1.1]{BC}) and so by Proposition \ref{no bad}, $W$ admits a block decomposition $\mathcal{B}$ with no bad tori. Thus we can split $W$ along a subset of its JSJ tori  into blocks $B_1, \ldots, B_n$ such that 
\vspace{-.2cm} 
\begin{itemize}

\item the pieces of $W$ are endowed with admissible foliations which determine a gluing coherent family of rational slopes on the JSJ tori of $W$; 

\vspace{.2cm} \item the admissible foliations piece together to form a smooth foliation $\mathcal{F}_i$ on each block $B_i$ which is hyperbolic on each component of $\partial B_i$.

\end{itemize}
We can assume that there is more than one block in $\mathcal{B}$ by \cite[Theorem 1.4]{OSz2004-genus}. 

Orient $W$ and co-orient $\mathcal{F}_1, \ldots, \mathcal{F}_n$ coherently. (The latter means that if $T = B_i \cap B_j$ is a JSJ torus of $W$, then the transverse orientation of a closed leaf of $\mathcal{F}_i|T$ and that of a closed leaf of $\mathcal{F}_j|T$ match, up to isotopy of one leaf to the other.) This determines an orientation of each $B_i$ and a coherent orientation of the leaves of each $\mathcal{F}_i$.

\begin{lemma} \label{omega_i} 
There are closed $2$-forms $\omega_1, \ldots, \omega_n$ on $W$ such that 

$(1)$ $\omega_i|\mathcal{F}_i > 0$. 

$(2)$ $\omega_i$ is zero outside of an arbitrarily small neighborhood of $B_i$ in $W$. 
\end{lemma}

\begin{proof}
This follows immediately from the construction on pages 157--158 of \cite{Ca}. 
\end{proof}

For each $i$, double the pair $(B_i, \mathcal{F}_i)$ to form a closed, oriented $3$-manifold $D(B_i)$ which admits a smooth, co-oriented taut foliation $D(\mathcal{F}_i)$. By \cite[Theorem 2.1.2]{ET}, we can approximate the tangent field of $D(\mathcal{F}_i)$ arbitrarily closely by a tight contact structure $\xi_i$ through a linear deformation. A genericity argument shows that we can assume that the characteristic foliation of $\xi_i$ is hyperbolic of the same slope as $\mathcal{F}_i$ on each boundary component of $B_i$. The restriction of $\xi_i$ to the positive copy of $B_i$ in $D(B_i)$, $\xi_i^+$ say, is positive on $B_i$. Its restriction to the negative copy, $\xi_i^-$, is negative with respect to the orientation on $B_i$ induced from $W$. Give $\xi_i^{\pm}$ the orientation induced from the leaves of $\mathcal{F}$. Then without loss of generality we can suppose that $\omega_i|\xi_i^{\pm} > 0$. 

Next we show how to build $\xi^+$ from $\xi_1^+, \ldots, \xi_n^+$. The construction of $\xi^-$ from $\xi_1^-, \ldots, \xi_n^-$ is similar and will be left to the reader. 

Thicken each JSJ torus of $W$ which is a boundary component of some $B_i$ and think of $W$ as being obtained by gluing together the blocks $B_i$ using these thickenings. More precisely, suppose that $B_i$ and $B_j$ share a JSJ torus $T$ and that $i < j$. We replace $B_i \cup_T B_j$ by $B_i \cup (T \times [0,1]) \cup B_j$ where $T \times \{0\}$ is identified with $T \subset \partial B_i$ and $T \times \{1\}$ with $T \subset \partial B_j$. 

\begin{lemma}
Fix a JSJ torus $T$ of $W$ which is a boundary component of $B_i$ and $B_j$. There is a contact structure $\xi_{ij}^+$ on 
$B_i \cup (T \times [0,1]) \cup B_j$ and a closed $2$-form $\omega_{ij}$ on $W$ such that 

$(1)$ $\xi_{ij}^+|T \times [0,1]$ is non-rotative, $\xi_{ij}^+|B_i = \xi_i^+$ and $\xi_{ij}^+|B_j = \xi_j^+$.

$(2)$ $\omega_{ij}|B_i = \omega_i$, $\omega_{ij}|B_j = \omega_i$, and $\omega_{ij}|\xi_{ij} > 0$. 
\end{lemma}

\begin{proof}
Without loss of generality we can suppose that $i < j$. First we extend $\xi_i$ across $T \times [0, \frac13]$. 

Let $\mathcal{G}$ denote the characteristic foliation on $T$ determined by $\xi_i^+$. Then $\mathcal{G}$ is hyperbolic with $2k$ closed leaves for some $k \geq 1$. Further, the dividing set $\Gamma_T$ of $T$ with respect to $\xi_i^+$ is a spine of the complement of these $2k$ closed leaves.  

Fix coordinates $(u,v) \in S^1 \times S^1$ on $T$ such that $\mathcal{G}$ is transverse to $\{u\} \times S^1$ for each $u \in S^1$, the closed leaves of $\mathcal{G}$ are each of the form $S^1 \times \{v\}$ for some $v \in S^1$ as are the components of the dividing set of $T$. Then there is a smooth function $g: T \to (-\frac{\pi}{2}, \frac{\pi}{2})$ such that the tangent space to $\mathcal{G}$ at $(u,v)$ is given by $\hbox{ker}(\sin(g(u,v))du - \cos(g(u,v))dv)$. Further, there are distinct $v_1, \ldots, v_{2k} \in S^1$ indexed successively around $S^1$ such that $g(u,v) = 0$ if and only if $v \in \{v_1, \ldots, v_{2k}\}$ and $g$ is alternately positive and negative on the open annuli $S^1 \times (v_i, v_{i+1})$.  

It follows from (the proof of) Giroux's Flexibility Theorem \cite[Proposition 3.6]{Gi} that we can extend $\xi_i^+$ across $B_i \cup_{T = T \times \{0\}} \big(T \times [0, \frac13]\big)$ in such a way that 
\vspace{-.2cm} 
\begin{itemize}

\item the extension is non-rotative on $T \times [0, \frac13]$ and is transverse to the circles to $\{u\} \times S^1 \times \{t\}$ for each $u \in S^1$ and $t \in [0, \frac13]$; 

\vspace{.2cm} \item each torus $T_t = T \times \{t\}$ is convex with dividing set $\Gamma_T \times \{t\}$; 

\vspace{.2cm} \item the characteristic foliation on $T_{\frac13}$ is hyperbolic with $2k$ closed leaves whose tangent line field is given by $\hbox{ker}(\sin(f_0(v))du - \cos(f_0(v))dv)$ where $f_0: S^1 \to  (-\frac{\pi}{2}, \frac{\pi}{2})$ is a smooth function which is positive, respectively negative, on $(v_i, v_{i+1})$ if $g$ is positive, respectively negative, on $S^1 \times (v_i, v_{i+1})$. Further, $f_0$ may be arbitrarily chosen subject to these conditions.. 
 
\end{itemize}
\vspace{-.2cm}
We claim that there is a closed $2$-form on $W$ which agrees with $\omega_i$ on $B_i$, is positive on the extension of $\xi_i^+$ over $T \times [0, \frac13]$, and is zero outside an arbitrarily small neighbourhood of $B_i \cup T \times [0, \frac13]$. To see this, assume, without loss of generality (cf. Lemma \ref{omega_i}(2)), that the support of each $\omega_i|B_i \cup T \times [0, \frac38]$ is contained in $B_i \cup T \times [0, \frac16]$. Let $\varphi: T \times [0,1] \to S^1 \times [0,1]$ be the projection $(u,v,t) \mapsto (u,t)$. For $\epsilon > 0$ arbitrarily small, choose a $2$-form $\eta$ on $S^1 \times \mathbb R$ which is non-zero on $S^1 \times (0,\frac13 + \epsilon)$ but identically zero on its complement. Then $\varphi^*(\eta)$ extends to a closed $2$-form $\zeta$ on $W$ which vanishes on the complement of $T \times (0,\frac13 + \epsilon)$. Further, since the extension of $\xi_i^+$ is transverse to each $\{u\} \times S^1 \times \{t\} \subset T \times [0, \frac13]$, we can suppose that $\zeta$ evaluates positively on this the restriction of the extension to $T \times [0, \frac13]$, at least up to replacing $\zeta$ by its negative. Then for all $r \gg 0$, 
$$\omega_i(r) = \omega_i + r \zeta$$ is the desired $2$-form. 

A similar construction extends $\xi_j^+$ appropriately across $B_j \cup_{T = T \times \{1\}} \big(T \times [\frac23, 1]\big)$ and produces a closed $2$-form $\omega_j(r)$ on $W$ (for large $r$) which agrees with $\omega_j$ on $B_j$, is positive on the extension of $\xi_j^+$, and is zero outside an arbitrarily small neighbourhood of $B_j \cup T \times [\frac23, 1]$. Further, the characteristic foliation on $T_{\frac23}$ is hyperbolic with the same number of closed leaves as the characteristic foliation $\mathcal{G}'$ on $T$ determined by $\xi_j^+$ and its tangent line field is given by $\hbox{ker}(\sin(f_1(v))du - \cos(f_1(v))dv)$ where $f_1: S^1 \to  (-\frac{\pi}{2}, \frac{\pi}{2})$ is a smooth function which is alternately positive and negative on the components of the complement of its zeros (as determined by $\mathcal{G}'$). Given the flexibility in our choice of coordinates on $T_1$ and, as remarked above, the  flexibility in the choice $f_1$, we can arrange for $f_0(v) < f_1(v)$ for all $v$. 

To complete the proof, we must interpolate between the extension of $\xi_i^+$ over $T \times [0, \frac13]$ and that of $\xi_j^+$ over $T \times [\frac23, 1]$ and construct the closed $2$-form $\omega_{ij}$.  

Set $f_t = (2-3t)f_0 + (3t -1)f_1$ and observe that as $f_0 < f_1$, we have $f_t(v) < f_s(v)$ if $t < s$. Thus the function $\varphi: T \times I \to T \times \mathbb R$ defined by $\varphi(u,v,t) = (u, v, f_t(v))$ is a smooth embedding. Endow $T \times \mathbb R$ with the contact structure determined by the contact $1$-form $\eta(u,v,t) = \sin(t)du - \cos(t)dv$. Next endow $T \times I$ with the contact structure $\xi_T$ determined by the contact $1$-form $\eta_0= \varphi^*(\eta)$. That is, 
$$\eta_0(u,v,t) = \sin(f_t(v))du - \cos(f_t(v))dv$$
It follows that the contact plane of $\xi_T$ at $(u,v,t)$ is spanned by the vectors $\{\frac{\partial}{\partial t}, \cos(f_t(v))\frac{\partial}{\partial u} + \sin(f_t(v))\frac{\partial}{\partial v} \}$ so 
\vspace{-.2cm}
\begin{itemize}

\item $\xi_T$ is transverse to each $T_t = T \times \{t\}$; 

\vspace{.2cm} \item the characteristic foliation on $T_t$ determined by $\xi_T$ is transverse to $\{u\} \times S^1 \times \{t\}$ for each $u$ and $t$; 

\vspace{.2cm} \item for each $t$ there are only finitely many closed orbits of the characteristic foliation of $\xi_T|T_t$ and each is of the form $S^1 \times \{v\} \times \{t\}$ for some $v \in S^1$;  

\vspace{.2cm} \item $\xi_T^+|T \times [\frac13,\frac23]$ is non-rotative. 

\end{itemize}
\vspace{-.2cm}
Use $\xi_T$ to piece together the extensions of $\xi_i^+$ and $\xi_j^+$, thus producing $\xi_{ij}^+$. As above we can construct a closed $2$-form $\omega_T$ on $W$ which is positive on $\xi_{ij}^+|T \times (\frac13, \frac23)$ but zero outside its complement. For $r \gg 0$, 
$$\omega_{ij}(r) = \omega_i(r) + r \omega_T + \omega_i(r)$$
satisfies (2) of the lemma. 
\end{proof}

Applying the lemma inductively to each boundary component of the blocks we can build a contact structure $\xi^+$ on $W$ and a $1$-parameter family of closed $2$-form $\omega(r)$ ($r \gg 0$) such that $\omega(r)|\xi^+ > 0$. We construct $\xi^-$ similarly, and the reader will verify that for large $r$, $\omega(r)|\xi^- > 0$. This completes the proof of Theorem \ref{LO implies NLS}.
\end{proof}

\end{document}